\documentclass[12pt]{article}
\usepackage{amsmath,amsthm,verbatim,amssymb,amsfonts,amscd, graphicx, mathrsfs}
\usepackage{diagrams}
\usepackage[usenames,dvipsnames]{color}
\usepackage{amstext}
\usepackage{color}

\theoremstyle{plain}
\newtheorem{theorem}{Theorem}[section]
\newtheorem{corollary}[theorem]{Corollary}
\newtheorem{lemma}[theorem]{Lemma}
\newtheorem{proposition}[theorem]{Proposition}

\theoremstyle{definition}
\newtheorem{definition}[theorem]{Definition}

\newtheorem*{question*}{Question}
\newtheorem*{questions*}{Questions}
\newtheorem*{theorem*}{Theorem}
\newtheorem*{corollary*}{Corollary}

\theoremstyle{remark}
\newtheorem{remark}[theorem]{Remark}
\newtheorem{example}[theorem]{Example}

\newcommand\blfootnote[1]{%
  \begingroup
  \renewcommand\thefootnote{}\footnote{#1}%
  \addtocounter{footnote}{-1}%
  \endgroup}

\newcommand{\wt}{\widetilde}
\newcommand{\xr}{\xrightarrow}

\newcommand{\td}[1]{\tilde{#1}}
\newcommand{\into}{\hookrightarrow}

\newcommand{\Z}{\mathbb{Z}}

\newcommand{\R}{\mathbb{R}}

\newcommand{\C}{\mathbb{C}}

\newcommand{\bd}{\partial}

\newcommand{\mc}[1]{\mathcal{#1}}
\newcommand{\ms}[1]{\mathscr{#1}}

\newcommand{\Ext}{\text{Ext}}
\newcommand{\mf}{\mathfrak}

\newcommand{\e}{\mathbf{e}}

\newcommand{\Aut}{\operatorname{Aut}}
\newcommand{\Hom}{\operatorname{Hom}}

\newcommand{\ran}{\operatorname{ran}}

\begin{document}

\title{Unitary equivalence of normal matrices over topological spaces} 
\author{Greg Friedman and Efton Park}
\date\today
\maketitle

\blfootnote{The first author was partially supported by National Science Foundation Grant DMS-1308306 and Simons Foundation Grant \#209127.}

\begin{abstract}
Let $A$ and $B$ be normal matrices with coefficients that are continuous complex-valued functions on a topological space $X$ 
 that has the homotopy type of a CW complex, and suppose these matrices have the same distinct eigenvalues at each point of $X$.  
We use obstruction theory to establish a necessary and sufficient condition for $A$ and $B$ to be unitarily equivalent. We also determine
bounds on the number of possible unitary equivalence classes in terms of cohomological invariants of $X$. 
\end{abstract}

\tableofcontents

\section{Introduction}
One of the most striking theorems in linear algebra is the spectral theorem: every normal matrix with
complex entries is diagonalizable.  An immediate consequence of the spectral theorem is that a 
normal matrix over $\mathbb{C}$ is determined up to unitary equivalence by its eigenvalues, counting multiplicities.  

Given the importance of the spectral theorem, it is natural to ask  whether it holds in more general situations.  
Suppose $X$ is a topological space.  Let $C(X)$ denote the 
$\mathbb{C}$-algebra of complex-valued continuous functions on $X$, and let $M_n(C(X))$ be the ring
of $n$-by-$n$ matrices with entries in $C(X)$.  By a slight abuse of terminology, we will refer to elements of
$M_n(C(X))$ as matrices over $X$.  Given $A$ in $M_n(C(X))$ and $x$ in $X$, we can evaluate at $x$
to obtain an element $A(x)$ of $M_n(\mathbb{C})$.  Define the adjoint of $A$ pointwise: $A^*(x) = (A(x))^*$.  
We can define normal matrices in $M_n(C(X))$ as those matrices that commute with their adjoint,
and we can also consider the set $U_n(C(X))$ of unitary matrices; that is, the set of matrices $U$ 
in $M_n(C(X))$ with the property that $UU^* = U^*U = I$.  Then two matrices $A,B\in M_n(C(X))$ are unitarily equivalent if there exists such a $U\in U_n(C(X))$ such that $B=U^*AU$, i.e. if $B(x)=U^*(x)A(x)U(x)$ for all $x\in X$. One can then ask the following question: 

\begin{question*}
Given a topological space $X$, what are the unitary equivalence classes of normal matrices in $M_n(C(X))$? In particular, is every such matrix diagonalizable, in which case there is only one equivalence class for each $n$?
\end{question*}

The question of diagonalizability has been considered before by previous authors. In \cite{Kad2}, R. Kadison gave an example of a
normal element of $M_2(C(S^4))$ that is 
not diagonalizable.   In \cite{GP}, K. Grove and G. K. Pedersen considered diagonalizability of matrices over compact Hausdorff spaces more generally.  
In that paper, they determined which compact Hausdorff spaces $X$ have the property
that every normal matrix over $X$ is diagonalizable.  Such topological spaces $X$ are rather exotic; for example,
no infinite first countable compact Hausdorff space has this property.   

The following simple example, which is a modification of Example 1.1 in \cite{GP}, illustrates one of the main obstructions to diagonalizability over more reasonable spaces.  Let $X$ be  $\mathbb{R}$ in its usual topology, and define
\[
A(x) = 
\begin{cases} 
\begin{pmatrix} x & 0 \\ 0 & 0 \end{pmatrix} \quad x \leq 0\\
\\
\begin{pmatrix} x & x \\ x & x \end{pmatrix} \quad x \geq 0.
\end{cases}
\]
The matrix $A(x)$ is normal for every real number $x$.  However, a direct calculation shows that there is no element $U$ in $U_2(C(\mathbb{R}))$ with the 
property that $U^*(x)A(x)U(x)$ is diagonal for all $x$.  Indeed, if such a $U$ existed, one of its columns would provided a continuously varying family of
eigenspaces associated to the eigenvalue $0$, and a close examination of $A$ shows that such a family cannot exist.  The real line $\mathbb{R}$ is contractible, so we see that lack of diagonalizability in this case cannot be detected by algebraic topological invariants.  Rather, the issue is that the multiplicity of the eigenvalue $0$ jumps at the origin.

By contrast, we will see in this paper that algebraic topology does, somewhat surprisingly given the analytic/algebraic nature of the problem, have something to say 
if we restrict our attention to
 \emph{multiplicity-free} normal matrices. A matrix $A\in M_n(C(X))$ is multiplicity free if, for each  $x$ in $X$, the eigenvalues of $A(x)$ are distinct. Grove and Pedersen showed that such matrices can be guaranteed to be diagonalizable over less exotic classes of spaces than those that are required for diagonalizability of all normal matrices. 
In fact, they proved \cite[Theorem 1.4]{GP} that if $X$ is a $2$-connected compact CW-complex, then every normal multiplicity-free matrix over $X$ is diagonalizable.
They also gave examples to show that the spectral theorem fails in general for multiplicity-free normal matrices over CW complexes that are not $2$-connected. 

Given this failure of diagonalizability, in general, even for multiplicity-free normal matrices,  we can  return to the more general part of our question, now restricted to multiplicity-free normal matrices,  and ask what we can say about the unitary equivalence classes\footnote{One immediate observation is that in order for $A, B\in M_n(C(X))$ to be unitarily equivalent, they must be unitarily equivalent over every $x\in X$ and so must have the same eigenvalues at every $x\in X$. In fact, $A$ and $B$ must have the same characteristic polynomials in $C(X)[\lambda]$, the ring of polynomials with coefficients in $C(X)$. It follows that no multiplicity-free normal matrix can be unitarily equivalent to a matrix that is not multiplicity free, and so multiplicity-free and non-multiplicity-free matrices really can be studied independently.}. 
As the above examples and results already demonstrate, and as will be borne out below, the multiplicity-free normal matrices provide a tractable class for exploration with a rich theory even on the reasonable class of spaces homotopy equivalent to CW complexes. In this setting, we will see that algebraic topology can be used as a tool to provide some answers to the following questions:

\paragraph{Questions:} 

\begin{enumerate}
\item Given two multiplicity-free normal matrices $A,B\in M_n(C(X))$, are $A$ and $B$ unitarily equivalent?

\item What can we say about the number of unitary equivalence classes of multiplicity-free normal matrices in $M_n(C(X))$?
\end{enumerate}

Our approach to these questions utilizes the algebraic topology notion of obstruction theory.  We begin by constructing a fiber bundle that encodes unitary equivalence information for matrices with
complex entries; i.e., matrices over a point.  Then, given normal multiplicity free matrices $A,B\in M_n(C(X))$  that have the same characteristic polynomial,
we associate to the matrices a continuous map from $X$ into the base of the fiber bundle, and we prove that $A$ and $B$ are unitarily equivalent if and only if
this map lifts to the total space.  We next construct a cohomology class $[\theta(A, B)]$ that lives in $H^2(X; \Pi_{A,B})$, where $\Pi_{A,B}$ is a system
of local coefficients determined by the monodromy of the eigenvalues of $A$ and $B$. This system $\Pi_{A,B}$  has fiber $\Z^n$, and the action of $[\gamma]\in \pi_1(X)$ permutes the $\Z$ factors according to the monodromy of the common eigenvalues of $A$ and $B$ as we travel around a loop representing $[\gamma]$. 
The cohomology class $[\theta(A, B)]$ is the complete obstruction to $A$ and $B$ being
unitarily equivalent. Specifically, we prove the following theorem (see Theorem \ref{T: U equiv} and Proposition \ref{P: nonCW ob}):

\begin{theorem*}
Let $X$ be (homotopy equivalent to) a CW complex, and let $A$ and $B$ be normal multiplicity-free matrices
in $M_n(C(X))$ that have the same characteristic polynomial.  Then there exists a  unique cohomology class 
$[\theta(A,B)] \in H^2(X; \Pi_{A,B})$ such that $A$ and $B$ are unitarily equivalent if and only if $[\theta(A,B)]=0$. 
\end{theorem*}

An immediate consequence of our theorem, which  is  not at all obvious from a strictly operator-theoretic perspective, is that if $X$ contains no $2$-cells and $A$ and $B$ are normal multiplicity
free matrices over $X$, then $A$ and $B$ are unitarily equivalent if and only if they have the same characteristic polynomial.
Another fairly direct consequence is a generalization of Grove and Pedersen's \cite[Theorem 1.4]{GP}; this is the theorem that states that if $X$ is a $2$-connected compact 
CW complex then any multiplicity-free normal matrix $A$ in $M_n(C(X))$ can be diagonalized. The following is our Corollary \ref{C: GP cor}:

\begin{corollary*}
Suppose that $X$ is a simply-connected (not necessarily compact) CW complex and that $\Hom(H_2(X),\Z)=0$ (in particular, when $H_2(X)$ is torsion).
Then any two normal multiplicity-free matrices $A$ and $B$ in $M_n(C(X))$ with the same eigenvalues at each point are unitarily equivalent. 
In particular, any normal multiplicity-free matrix in $M_n(C(X))$ is diagonalizable.
\end{corollary*}

Less obviously, our obstruction also begins to provide answers to our second question, concerning the number of unitary equivalence classes of multiplicity-free normal matrices over $X$. In Section \ref{S: OR}, we demonstrate the following as Corollary \ref{C: counting}, slightly rephrased here for the introduction:

\begin{corollary*}
Given a connected CW complex $X$ and a multiplicity-free  polynomial $\mu\in C(X)[\lambda]$, the number of unitary equivalence classes of normal matrices in $M_n(C(X))$ with 
characteristic polynomial $\mu$ is less than or equal to the cardinality of $H^2(X; \Z^n_\rho)$, where $\Z^n_\rho$ is the system of local coefficients with fiber $\Z^n$ and representation of $\pi_1(X)$  determined by the monodromy of the roots of $\mu$.
In particular, if $H^2(X; \Z^n_\rho)$ is finite, there are a finite number of such equivalence classes, and if $X$ contains a countable number of cells, there are 
a countable number of such equivalence classes.
\end{corollary*}

Even the final statement that if $X$ has a countable number of cells  then there are a countable number of 
unitary equivalence classes is not obvious; \emph{a priori}, there could be an uncountable number of equivalence classes.   

\paragraph{Organization.}
The paper is organized as follows: In Section 2, we construct, for each natural number $n$, an $n$-torus fiber bundle $p: E_n \longrightarrow B_n$;
this bundle captures information about various ways one set of one-dimensional orthogonal spanning projections can be unitarily conjugated to another set.  
In Section 3, we show that given two normal multiplicity free matrices $A$ and $B$ over $X$ that have the same characteristic polynomial, 
there is a continuous map $\Phi_{A,B}: X \longrightarrow B_n$ with the feature that $A$ and $B$ are unitarily equivalent if and only if $\Phi_{A,B}$ lifts
to a map to $E_n$.    By replacing the unitary equivalence question into one involving the lifting of maps, we establish the aforementioned theorem and corollary.  
 In Section 4, we explore the functorial and naturality properties of our invariant, and extend $[\theta(A, B)]$ to certain topological spaces that are not CW complexes.  In Section 5, we examine monodromy issues and show that the coefficient system $\Pi_{A,B}$ 
only depends on the common characteristic polynomial of $A$ and $B$, not on the matrices themselves.  In Section 6, we consider how $[\theta(A,B)]$ 
behaves when we vary $A$ and $B$, and we also explore how $[\theta(A,B)]$, $[\theta(B,C)]$, and $[\theta(A,C)]$ are related when $A$, $B$, and $C$ are
normal multiplicity free matrices with the same characteristic polynomial. This leads to our bounds on the cardinality of the set of unitary equivalence classes.  In Section 7, we show that if the characteristic polynomial globally factors into 
linear factors, then we can write our invariant in terms of Chern classes, and we look at some examples.  In the final section, we close with some open questions.

\section{A useful fiber bundle}

We construct a fiber bundle $p:E_n\longrightarrow B_n$, starting with the base.  Let $\mathcal{P}$ and $\mathcal{Q}$ be sets of $n$ pairwise orthogonal
projections in $M_n(\mathbb{C})$; it is important to observe that we do not assume any ordering of the elements of $\mathcal{P}$ and $\mathcal{Q}$. 
Note that each projection in $\mathcal{P}$ has rank one and that $\sum_{P \in \mathcal{P}} P$ is the identity matrix $I_n$.  Similarly, each
projection in $\mathcal{Q}$ has rank one and $\sum_{Q \in \mathcal{Q}} Q = I_n$.  Set
\[
B_n = \bigl\{(\mathcal{P}, \mathcal{Q}, \sigma) : \text{ $\sigma$ is a bijection from $\mathcal{P}$ to $\mathcal{Q}$}\bigr\}.
\]

We will construct a metric on $B_n$. Let $\Vert \cdot \Vert_2$ be the usual Hilbert space norm on $\mathbb{C}^n$; i.e., if 
$\{\e_i\}$ is an orthonormal basis of  $\mathbb{C}^n$ in its standard inner product
and $v = \sum_{i=1}^n \lambda_i\e_i$, then $\Vert v \Vert_2 = \sqrt{\sum_{i=1}^n |\lambda_i |^2}$.
Let $\Vert \cdot \Vert$ denote the operator norm on $M_n(\mathbb{C})$:
\[
\Vert A \Vert = \sup \left\{\frac{\Vert Av \Vert_2}{\Vert v \Vert_2} : v \neq 0\right\}.
\]
For each pair of elements of $B_n$, define 
\begin{multline*}
d\bigl((\mathcal{P}, \mathcal{Q}, \sigma), (\widetilde{\mathcal{P}}, \widetilde{\mathcal{Q}}, \widetilde{\sigma})\bigr) = \\
\min\Bigl\{ \max\Bigl\{\Vert P - \tau(P)\Vert, \Vert \sigma(P) - \widetilde{\sigma}\tau(P)\Vert : P\in\mathcal{P}\Bigr\}
: \text{ $\tau$ is a bijection from $\mathcal{P}$ to $\widetilde{\mathcal{P}}$}\Bigr\}.
\end{multline*}
Roughly speaking, the idea of the definition is that we measure the distance between sets of projections by looking at the distances among individual pairs of 
projections after using $\tau$ to match up the pairs as closely as possible.

\begin{proposition} The function $d$ is a metric (distance function) on $B_n$.
\end{proposition}

\begin{proof}
Clearly $d\bigl((\mathcal{P}, \mathcal{Q}, \sigma), (\widetilde{\mathcal{P}}, \widetilde{\mathcal{Q}}, \widetilde{\sigma})\bigr)$
is always nonnegative; suppose this quantity equals $0$.  Then there exists a bijection 
$\tau: \mathcal{P} \longrightarrow \widetilde{\mathcal{P}}$ with the property that $P = \tau(P)$ for every $P$ in $\mathcal{P}$.
Thus $\mathcal{P} = \widetilde{\mathcal{P}}$ and $\tau$ is the identity map.  Next,
$\sigma(P) = \widetilde{\sigma}\tau(P) = \widetilde{\sigma}(P)$ for all $P$ in $\mathcal{P}$, so  $\sigma = \widetilde{\sigma}$ and thus $\mathcal{Q} = \widetilde{\mathcal{Q}}$.

Next, let $(\mathcal{P}, \mathcal{Q}, \sigma)$ and $(\widetilde{\mathcal{P}}, \widetilde{\mathcal{Q}}, \widetilde{\sigma})$ be
arbitrary elements of $B_n$ and choose 
$\tau: \mathcal{P} \longrightarrow \widetilde{\mathcal{P}}$ so that the minimum in 
$d\bigl((\mathcal{P}, \mathcal{Q}, \sigma), (\widetilde{\mathcal{P}}, \widetilde{\mathcal{Q}}, \widetilde{\sigma})\bigr)$ 
is realized.  Then
\begin{align*}
d\bigl((\mathcal{P}, \mathcal{Q}, \sigma), (\widetilde{\mathcal{P}}, \widetilde{\mathcal{Q}}, \widetilde{\sigma})\bigr) 
&= \max\Bigl\{\Vert P - \tau(P)\Vert, \Vert \sigma(P) - \widetilde{\sigma}\tau(P)\Vert : P\in\mathcal{P}\Bigr\} \\
&= \max\Bigl\{\Vert \tau^{-1}(\widetilde{P}) - \widetilde{P}\Vert, 
\Vert \sigma\tau^{-1}(\widetilde{P}) - \widetilde{\sigma}(\widetilde{P})\Vert : \widetilde{P} \in \widetilde{\mathcal{P}} \Bigr\} \\
&= \max\Bigl\{\Vert \widetilde{P} - \tau^{-1}(\widetilde{P})\Vert, 
\Vert \widetilde{\sigma}(\widetilde{P}) - \sigma\tau^{-1}(\widetilde{P}) \Vert : \widetilde{P} \in \widetilde{\mathcal{P}}\Bigr\} \\
&\geq d\bigl((\widetilde{\mathcal{P}}, \widetilde{\mathcal{Q}}, \widetilde{\sigma}), (\mathcal{P}, \mathcal{Q}, \sigma)\bigr).
\end{align*}
Reversing the roles of $(\mathcal{P}, \mathcal{Q}, \sigma)$ and 
$(\widetilde{\mathcal{P}}, \widetilde{\mathcal{Q}}, \widetilde{\sigma})$ establishes the symmetry of $d$.

 Finally, for three arbitrary elements $(\mathcal{P}, \mathcal{Q}, \sigma)$, 
 $(\widetilde{\mathcal{P}}, \widetilde{\mathcal{Q}}, \widetilde{\sigma})$,
 and  $(\widehat{\mathcal{P}}, \widehat{\mathcal{Q}}, \widehat{\sigma})$, in $B_n$, 
 choose $\tau: \mathcal{P} \longrightarrow \widetilde{\mathcal{P}}$ and 
 $\nu: \widetilde{\mathcal{P}} \longrightarrow \widehat{\mathcal{P}}$ so that the minima in the definitions of 
 $d\bigl((\mathcal{P}, \mathcal{Q}, \sigma), (\widetilde{\mathcal{P}}, \widetilde{\mathcal{Q}}, \widetilde{\sigma})\bigr)$ 
 and 
 $d\bigl((\widetilde{\mathcal{P}}, \widetilde{\mathcal{Q}}, \widetilde{\sigma}), 
 (\widehat{\mathcal{P}}, \widehat{\mathcal{Q}}, \widehat{\sigma})\bigr)$ are realized.  Then
 \begin{align*}
&d\bigl((\mathcal{P}, \mathcal{Q}, \sigma), (\widetilde{\mathcal{P}}, \widetilde{\mathcal{Q}}, \widetilde{\sigma})\bigr) + 
d\bigl((\widetilde{\mathcal{P}}, \widetilde{\mathcal{Q}}, \widetilde{\sigma}), 
 (\widehat{\mathcal{P}}, \widehat{\mathcal{Q}}, \widehat{\sigma})\bigr) \\
&= \max\Bigl\{\Vert P - \tau(P)\Vert, \Vert \sigma(P) - \widetilde{\sigma}\tau(P)\Vert : P\in\mathcal{P}\Bigr\} \\
&{\hskip180pt}+  \max\Bigl\{\Vert \widetilde{P} - \nu(\widetilde{P})\Vert, 
\Vert \widetilde{\sigma}(\widetilde{P}) - \widehat{\sigma}\nu(\widetilde{P})\Vert : 
\widetilde{P} \in \widetilde{\mathcal{P}}\Bigr\} \\
&\geq \max\Bigl\{\Vert P - \tau(P)\Vert + \Vert \widetilde{P} - \nu(\widetilde{P})\Vert,
\Vert \sigma(P) - \widetilde{\sigma}\tau(P)\Vert + \Vert \widetilde{\sigma}(\widetilde{P}) - \widehat{\sigma}\nu(\widetilde{P})\Vert :
P\in\mathcal{P}, \widetilde{P} \in \widetilde{\mathcal{P}}\Bigr\} \\
&\geq \max\Bigl\{\Vert P - \tau(P)\Vert + \Vert \tau(P) - \nu\tau(P)\Vert, \Vert \sigma(P) - \widetilde{\sigma}\tau(P)\Vert + 
 \Vert \widetilde{\sigma}\tau(P) - \widehat{\sigma}\nu\tau(P)\Vert :  P\in\mathcal{P} \Bigr\}\\
&\geq \max\Bigl\{\Vert P - \nu\tau(P)\Vert, \Vert \sigma(P) - \widehat{\sigma}\nu\tau(P)\Vert :  P\in\mathcal{P}\Bigr\}\\
&\geq d\bigl((\mathcal{P}, \mathcal{Q}, \sigma), (\widehat{\mathcal{P}}, \widehat{\mathcal{Q}}, \widehat{\sigma})\bigr).
\end{align*}
\end{proof}
Endow $B_n$ with the metric topology associated to $d$.  

\begin{lemma}\label{technical}
Let $(\mathcal{P}, \mathcal{Q}, \sigma)$, $(\widetilde{\mathcal{P}}, \widetilde{\mathcal{Q}}, \widetilde{\sigma})$,
and $(\widehat{\mathcal{P}}, \widehat{\mathcal{Q}}, \widehat{\sigma})$ be elements of $B_n$.
\begin{enumerate}
\item Suppose there exists a bijection $\widetilde{\tau}: \mathcal{P} \longrightarrow \widetilde{\mathcal{P}}$ with the property that 
\[
\max\Bigl\{\Vert P - \widetilde{\tau}(P)\Vert, \Vert \sigma(P) - \widetilde{\sigma}\widetilde{\tau}(P)\Vert : P\in\mathcal{P}\Bigr\} <  \frac{1}{2}.
\]
Then 
\[
d\bigl((\mathcal{P}, \mathcal{Q}, \sigma), (\widetilde{\mathcal{P}}, \widetilde{\mathcal{Q}}, \widetilde{\sigma})\bigr) 
= \max\Bigl\{\Vert P - \widetilde{\tau}(P)\Vert, \Vert \sigma(P) - \widetilde{\sigma}\widetilde{\tau}(P)\Vert : P\in\mathcal{P}\Bigr\}.
\]
In other words, $\widetilde{\tau}$ realizes the minimum in the definition of 
$d\bigl((\mathcal{P}, \mathcal{Q}, \sigma), (\widetilde{\mathcal{P}}, \widetilde{\mathcal{Q}}, \widetilde{\sigma})\bigr)$. 
Furthermore, $\widetilde{\tau}$ is the unique bijection with this property.

\item Suppose that 
$d\bigl((\mathcal{P}, \mathcal{Q}, \sigma), (\widetilde{\mathcal{P}}, \widetilde{\mathcal{Q}}, \widetilde{\sigma})\bigr)$
and $d\bigl((\mathcal{P}, \mathcal{Q}, \sigma), (\widehat{\mathcal{P}}, \widehat{\mathcal{Q}}, \widehat{\sigma})\bigr)$
are less than $1\slash 4$
and let $\widetilde{\tau}$ and $\widehat{\tau}$ be the bijections that realize the minima 
for $d$ in these two cases, respectively.  If
\[
d\bigl((\widetilde{\mathcal{P}}, \widetilde{\mathcal{Q}}, \widetilde{\sigma}), 
(\widehat{\mathcal{P}}, \widehat{\mathcal{Q}}, \widehat{\sigma})\bigr) < \epsilon,
\]
then $\Vert \widetilde{\tau}(P) - \widehat{\tau}(P)\Vert < \epsilon$ for all $P$ in $\mathcal{P}$.
\end{enumerate}
\end{lemma}

\begin{proof}(i). From the definition of $d$, we see that $\Vert P - \widetilde{\tau}(P)\Vert < 1\slash2$ for every
$P$ in $\mathcal{P}$.  Select one such $P$ and let $\widetilde{P}$ be any
element of $\widetilde{\mathcal{P}}$ other than $\widetilde{\tau}(P)$.  The ranges of the elements of $\widetilde{\mathcal{P}}$
are pairwise orthogonal and span $\mathbb{C}^n$, whence $\ran \widetilde{P} \subseteq \ran (\tau(P))^\perp$.
Therefore for any unit vector $v$ in $\ran \widetilde{P}$, 
\[
\bigl(\widetilde{P} - \widetilde{\tau}(P)\bigr)v = \widetilde{P}v - \widetilde{\tau}(P)v = \widetilde{P}v = v,
\]
and thus $\Vert \widetilde{P}- \widetilde{\tau}(P)\Vert \geq 1$.  The triangle inequality then yields
\[
\Vert P - \widetilde{P}\Vert \geq \Vert \widetilde{P} - \widetilde{\tau}(P)\Vert - \Vert P - \widetilde{\tau}(P)\Vert > 1 - \frac{1}{2} = \frac{1}{2},
\]
and hence any other choice of bijection from $\mathcal{P}$ to $\widetilde{\mathcal{P}}$ will not achieve the minimum
in the definition of $d\bigl((\mathcal{P}, \mathcal{Q}, \sigma), (\widetilde{\mathcal{P}}, \widetilde{\mathcal{Q}}, \widetilde{\sigma})\bigr)$.
\newline
(ii).  Because $\widetilde{\tau}$ is a bijection,
\begin{align*}
\max\Bigl\{\Vert \widetilde{P} - \widehat{\tau}\widetilde{\tau}^{-1}(\widetilde{P})\Vert : \widetilde{P} \in \widetilde{\mathcal{P}}\Bigr\} 
&\leq \max\Bigl\{\Vert \widetilde{P} - \widetilde{\tau}^{-1}(\widetilde{P})\Vert + 
\Vert \widetilde{\tau}^{-1}(\widetilde{P}) - \widehat{\tau}\widetilde{\tau}^{-1}(\widetilde{P})\Vert: \widetilde{P} \in \widetilde{\mathcal{P}}\Bigr\} \\
&\leq \max\Bigl\{\Vert \widetilde{\tau}(P) - P\Vert + \Vert P - \widehat{\tau}(P)\Vert: P \in \mathcal{P}\Bigr\} \\
&< \frac{1}{4} + \frac{1}{4} = \frac{1}{2}
\end{align*}
and 
\begin{align*}
\max\Bigl\{\Vert \widetilde{\sigma}(\widetilde{P}) - \widehat{\sigma}\widehat{\tau}\widetilde{\tau}^{-1}(\widetilde{P})\Vert 
&: \widetilde{P} \in \widetilde{\mathcal{P}}\Bigr\} \\
&\leq \max\Bigl\{\Vert \widetilde{\sigma}(\widetilde{P}) - \sigma\widetilde{\tau}^{-1}(\widetilde{P})\Vert 
+ \Vert \sigma\widetilde{\tau}^{-1}(\widetilde{P}) - \widehat{\sigma}\widehat{\tau}\widetilde{\tau}^{-1}(\widetilde{P})\Vert 
: \widetilde{P} \in \widetilde{\mathcal{P}}\Bigr\} \\
&\leq \max\Bigl\{\Vert \widetilde{\sigma}\widetilde{\tau}(P) - \sigma(P)\Vert 
+ \Vert \sigma(P) - \widehat{\sigma}\widehat{\tau}(P)\Vert : P \in \mathcal{P}\Bigr\} \\
&< \frac{1}{4} + \frac{1}{4} = \frac{1}{2}.
\end{align*}
From (i) we see that $\widehat{\tau}\widetilde{\tau}^{-1}$ is 
the bijection that realizes the minimum in 
$d\bigl((\widetilde{\mathcal{P}}, \widetilde{\mathcal{Q}}, \widetilde{\sigma}), 
(\widehat{\mathcal{P}}, \widehat{\mathcal{Q}}, \widehat{\sigma})\bigr)$.  Thus
\[ 
\epsilon > \max\{\Vert \widetilde{P} - \widehat{\tau}\widetilde{\tau}^{-1}(\widetilde{P})\Vert : \widetilde{P} \in \widetilde{\mathcal{P}}\}
= \max\{\Vert \widetilde{\tau}(P) - \widehat{\tau}(P)\Vert : P \in \mathcal{P}\}.
\]
\end{proof}

Endow $M_n(\mathbb{C})$ with its usual topology, and let $U_n$ be the topological subspace of unitary matrices in $M_n(\mathbb{C})$.  Define
\[
E_n = \bigl\{\bigl((\mathcal{P}, \mathcal{Q}, \sigma), U\bigr) \in B_n \times U_n : UPU^* = \sigma(P)\  
\text{for all $P$ in $\mathcal{P}$}\bigr\}.
\]
Note that $\bigl((\mathcal{P}, \mathcal{Q}, \sigma), U\bigr)$ is in $E_n$ if and only if $U$ restricts to an isometric vector space isomorphism 
from $\ran P$ to $\ran\sigma(P)$ for every $P$ in $\mathcal{P}$.  

Equip $E_n$ with the subspace topology it inherits from $B_n \times U_n$, and let 
$p: E_n \longrightarrow B_n$ be the projection map.  

If $\bigl((\mathcal{P}, \mathcal{Q}, \sigma), U\bigr)$ and $((\mathcal{P}, \mathcal{Q}, \sigma), \widetilde{U})$ are both in $E_n$, then they both
lie in $p^{-1}((\mathcal{P}, \mathcal{Q}, \sigma))$, and the unitaries $U$ and $\wt U$ each restrict to isometries from $\ran P$ to $\ran\sigma(P)$ for
every $P$ in  $\mathcal{P}$. As $\ran P$ and $\ran\sigma(P)$ are both one-dimensional subspaces of $\C^n$, two such isometries can differ
from each other only by an isometry of $\C$; such isometries can be represented by elements of $S^1$. Furthermore, because
$\{\ran P\}_{P\in \mc P}$ is a basis of $\C^n$, the matrices $U$ and $\wt U$ are determined completely by these one-dimensional isometries. 
Therefore, roughly speaking, the difference between $U$ and  $\wt U$ can be quantified by an element of 
$T^n \cong \prod_{P\in\mathcal{P}} S^1$. This is part of the content of the following, more precise, statement.

\begin{proposition}\label{fiber structure}
If $\bigl((\mathcal{P}, \mathcal{Q}, \sigma), U\bigr)$ is in $E_n$, then $((\mathcal{P}, \mathcal{Q}, \sigma), \widetilde{U})$
is in $p^{-1}((\mathcal{P}, \mathcal{Q}, \sigma))$ if and only if 
\[
\widetilde{U} = \sum_{P\in\mathcal{P}}\widetilde{z}_P\sigma(P)UP
\]
for some set $\{\widetilde{z}_P\}$ of complex numbers of modulus $1$.  
Furthermore, each such $\widetilde{U}$ can be uniquely written in this form.
\end{proposition}

\begin{proof} Suppose $\widetilde{U}$ has the form described in the statement of the proposition.  From the definition of $E_n$, we have
$\sigma(P) = UPU^*$ for all $P$ in $\mathcal{P}$.  The projections $\sigma(P)$ in $\mathcal{Q}$ are pairwise orthogonal, and thus
\begin{align*}
\widetilde{U}\widetilde{U}^* &= \left(\sum_{P \in \mathcal{P}} \widetilde{z}_P\sigma(P)UP\right)
\left(\sum_{P \in \mathcal{P}}\overline{\widetilde{z}}_P P^*U^*\sigma(P)^*\right) \\
&= \left(\sum_{P \in \mathcal{P}} \widetilde{z}_P\sigma(P)UP\right)
\left(\sum_{P \in \mathcal{P}}\overline{\widetilde{z}}_P PU^*\sigma(P)\right) \\
&= \sum_{P \in \mathcal{P}} \widetilde{z}_P\overline{\widetilde{z}}_P \sigma(P)UPU^*\sigma(P) \\
&= \sum_{P \in \mathcal{P}}  \sigma(P)\sigma(P)\sigma(P) \\
&= \sum_{P \in \mathcal{P}}  \sigma(P) \\
&= I.
\end{align*}
A similar computation establishes that $\widetilde{U}^*\widetilde{U} = I$, so $\widetilde{U}$ is unitary.  
Next, because the projections in $\mathcal{P}$ are also pairwise orthogonal, we see that
\[
\widetilde{U}P = \widetilde{z}_P\sigma(P)UP =\sigma(P)\widetilde{U},
\]
and hence $\widetilde{U}P\widetilde{U}^* = \sigma(P)$ for every $P$ in $\mathcal{P}$.  The uniqueness of the representation
of $\widetilde{U}$ in the desired form is evident.

Now suppose that $\bigl((\mathcal{P}, \mathcal{Q}, \sigma), \widehat{U}\bigr)$ is in $E_n$.
Fix $P$ in $\mathcal{P}$.  From the remarks following the definition of $E_n$, both $U$ and $\widehat{U}$ restrict to isometric vector space
isomorphisms from $\ran P$ to $\ran\sigma(P)$; in symbols, these isomorphisms are $\sigma(P)UP$ and $\sigma(P)\widehat{U}P$.
The subspaces $\ran P$ and $\ran\sigma(P)$ are one-dimensional, so we must have 
$\sigma(P)\widehat{U}P = \widehat{z}_P\sigma(P)UP$ for some complex number $\widehat{z}_P$ of modulus $1$.  This holds true
for every $P$ in $\mathcal{P}$, and the pairwise orthogonality of the projections in $\mathcal{P}$ and $\mathcal{Q}$
implies that 
\[
\widehat{U} = \sum_{P\in\mathcal{P}}\widehat{z}_P\sigma(P)UP,
\]
whence $\widehat{U}$ has the claimed form.
\end{proof}

A consequence of Proposition \ref{fiber structure} is that we can identify
$p^{-1}((\mathcal{P}, \mathcal{Q}, \sigma))$ with $T^n \cong \prod_{P\in\mathcal{P}} S^1$.  In fact, $E_n$ is a $T^n$-fiber
bundle over $B_n$.  To show this, we first need to establish a technical result.

\begin{lemma}\label{close} Let $P$ and $\widetilde{P}$ be projections in $M_n(\mathbb{C})$ and suppose
that $\Vert P - \widetilde{P}\Vert < 1$.  Then $I + \widetilde{P} - P$ maps $\ran P$ isomorphically onto $\ran\widetilde{P}$.
\end{lemma}

\begin{proof} The matrix $I + \widetilde{P} - P$ is invertible by Proposition 1.3.4 in \cite{Park}.  Take $v$ in $\ran P$.
Then $Pv = v$, and because $P^2 = P$, we see that
\[
(I + \widetilde{P} - P)v = (I + \widetilde{P} - P)Pv = Pv + \widetilde{P}Pv - Pv =  \widetilde{P}Pv.
\]
Therefore $I + \widetilde{P} - P$ is an injective vector space homomorphism from $\ran P$ to $\ran\widetilde{P}$,
which implies that $\dim\ran P \leq \dim\ran\widetilde{P}$.  A similar computation shows that $I + P - \widetilde{P}$ is
an injective vector space homomorphism from $\ran\widetilde{P}$ to $\ran P$, whence 
$\dim\ran\widetilde{P} \leq \dim\ran P$.  Thus $\dim\ran P = \dim\ran\widetilde{P}$ and $I + \widetilde{P} - P$ 
is an isomorphism from $\ran P$ to $\ran\widetilde{P}$.
\end{proof}

 \begin{proposition}\label{fiber} For each natural number $n$, the map $p$ makes $E_n$ into a fiber bundle 
 over $B_n$ with fiber homeomorphic to $T^n$, the $n$-dimensional torus.
 \end{proposition}
 
 \begin{proof}
 Fix an element $\bigl((\mathcal{P}, \mathcal{Q}, \sigma), U\bigr)$ of $E_n$.  For each $P$ in $\mathcal{P}$,
 choose unit vectors $v_P$ and $w_P$ in $\ran P$ and $\ran \sigma(P)$ respectively.  Set 
\[
 \mathcal{O} = \bigl\{ (\widetilde{\mathcal{P}}, \widetilde{\mathcal{Q}}, \widetilde{\sigma}) \in B_n : 
 d\bigl((\mathcal{P}, \mathcal{Q}, \sigma), 
 (\widetilde{\mathcal{P}}, \widetilde{\mathcal{Q}}, \widetilde{\sigma})\bigr) < 1\slash4\bigr\} 
 \]
and take  
$\bigl((\widetilde{\mathcal{P}}, \widetilde{\mathcal{Q}}, \widetilde{\sigma}), \widetilde{U}\bigr)$ in $p^{-1}(\mathcal{O})$.
Let $\widetilde{\tau}: \mathcal{P} \longrightarrow \widetilde{\mathcal{P}}$ be the bijection that realizes the minimum for 
 $d\bigl((\mathcal{P}, \mathcal{Q}, \sigma),  (\widetilde{\mathcal{P}}, \widetilde{\mathcal{Q}}, \widetilde{\sigma})\bigr)$.
Lemma \ref{close} shows that $I + \widetilde{\tau}(P) - P$ maps $\ran P$ isomorphically onto
$\ran \widetilde{\tau}(P)$ and $I + \widetilde{\sigma}\widetilde{\tau}(P) - \sigma(P)$ maps $\ran \sigma(P)$
isomorphically onto $\ran \widetilde{\sigma}\widetilde{\tau}(P)$ for every $P$ in $\mathcal{P}$.  In particular,
$(I + \widetilde{\tau}(P) - P)v_P$ and $(I + \widetilde{\sigma}\widetilde{\tau}(P) - \sigma(P))w_P$ are nonzero. For each
$P$ in $\mathcal{P}$, the complex vector spaces $\ran \widetilde{\tau}(P)$ and $\ran \widetilde{\sigma}\widetilde{\tau}(P)$ are
one-dimensional, and so $\widetilde{U}$ maps $\ran \widetilde{\tau}(P)$ isomorphically to $\ran \widetilde{\sigma}\widetilde{\tau}(P)$.  
Furthermore, unitary matrices map unit vectors to unit vectors, so for each $P$ in $\mathcal{P}$, the quantity
\[
z_{\widetilde{\tau}, P} = 
\left\langle \widetilde{U}\left(\frac{(I + \widetilde{\tau}(P) - P)v_P}{\Vert(I + \widetilde{\tau}(P) - P)v_P\Vert}\right),
\frac{(I + \widetilde{\sigma}\widetilde{\tau}(P) - \sigma(P))w_P}{\Vert (I + \widetilde{\sigma}\widetilde{\tau}(P) - \sigma(P))w_P\Vert}\right\rangle
\]
has modulus $1$.   Write $T^n$ as $\prod_{P \in \mathcal{P}}S^1$ and define
$\phi: p^{-1}(\mathcal{O}) \longrightarrow \mathcal{O} \times T^n$ by
\[
\phi\bigl((\widetilde{\mathcal{P}}, \widetilde{\mathcal{Q}}, \widetilde{\sigma}), \widetilde{U}\bigr) = 
\left((\widetilde{\mathcal{P}}, \widetilde{\mathcal{Q}}, \widetilde{\sigma}), \bigoplus_{P\in\mathcal{P}}z_{\widetilde{\tau}, P}\right).
\]
To show that $\phi$ is continuous, it clearly suffices to prove that the map 
$\bigl((\widetilde{\mathcal{P}}, \widetilde{\mathcal{Q}}, \widetilde{\sigma}), \widetilde{U}\bigr) \mapsto z_{\widetilde{\tau}, P}$ 
is continuous for each $P$ in $\mathcal{P}$.  Define $\Phi_P: p^{-1}(\mathcal{O}) \longrightarrow \mathbb{C}^n$ by the formula
$\Phi_P\bigl((\widetilde{\mathcal{P}}, \widetilde{\mathcal{Q}}, \widetilde{\sigma}), \widetilde{U}\bigr) = (I + \widetilde{\tau}(P) - P)v_P$.
Suppose that $\bigl((\widetilde{\mathcal{P}}, \widetilde{\mathcal{Q}}, \widetilde{\sigma}), \widetilde{U}\bigr)$ and 
$\bigl((\widehat{\mathcal{P}}, \widehat{\mathcal{Q}}, \widehat{\sigma}), \widehat{U}\bigr)$ are in $p^{-1}(\mathcal{O})$ and that
$d\bigl((\widetilde{\mathcal{P}}, \widetilde{\mathcal{Q}}, \widetilde{\sigma}), 
(\widehat{\mathcal{P}}, \widehat{\mathcal{Q}}, \widehat{\sigma})\bigr) < \epsilon$.  Using the result of, as well as the notation from,
Lemma \ref{technical}(ii), we obtain
\begin{align*}
\Vert \Phi_P\bigl((\widetilde{\mathcal{P}}, \widetilde{\mathcal{Q}}, \widetilde{\sigma}), \widetilde{U}\bigr) -
\Phi_P\bigl((\widehat{\mathcal{P}}, \widehat{\mathcal{Q}}, \widehat{\sigma}), \widehat{U}\bigr)\Vert
&= \Vert (I + \widetilde{\tau}(P) - P)v_P - (I + \widehat{\tau}(P) - P)v_P\Vert \\
&= \Vert (\widetilde{\tau}(P) - \widehat{\tau}(P))v_P\Vert \\
&\leq \Vert \widetilde{\tau}(P) - \widehat{\tau}(P)\Vert \\
&< \epsilon,
\end{align*}
and so each $\Phi_P$ is continuous.  The formula for each $z_{\widetilde{\tau}, P}$ 
is therefore a composition of continuous functions, and thus the map
$\bigl((\widetilde{\mathcal{P}}, \widetilde{\mathcal{Q}}, \widetilde{\sigma}), \widetilde{U}\bigr) \mapsto z_{\widetilde{\tau}, P}$ 
is continuous.

Next, define $\psi: \mathcal{O} \times T^n \longrightarrow p^{-1}(\mathcal{O})$ in the following way:
take $(\widetilde{\mathcal{P}}, \widetilde{\mathcal{Q}}, \widetilde{\sigma})$ in $\mathcal{O}$ and let
$\widetilde{\tau}$, $v_P$, and $w_P$ be as above.  Suppose 
\[
\left((\widetilde{\mathcal{P}}, \widetilde{\mathcal{Q}}, \widetilde{\sigma}), 
\bigoplus_{P\in\mathcal{P}}\zeta_P\right)
\]
is in $\mathcal{O} \times T^n$.  The set of vectors
\[
\left\{\frac{(I + \widetilde{\tau}(P) - P)v_P}{\Vert(I + \widetilde{\tau}(P) - P)v_P\Vert} : P \in \mathcal{P}\right\} =
\left\{\frac{(I + \widetilde{P} - \widetilde{\tau}^{-1}(\widetilde{P}))v_{\widetilde{\tau}^{-1}(\widetilde{P})}}
{\Vert((I + \widetilde{P} - \widetilde{\tau}^{-1}(\widetilde{P}))v_{\widetilde{\tau}^{-1}(\widetilde{P})}\Vert} : \widetilde{P} \in \widetilde{\mathcal{P}}\right\}
\]
spans $\mathbb{C}^n$, so we can define a unitary matrix $\widetilde{U}$ by setting
\[
\widetilde{U}\left(\frac{(I + \widetilde{P} - \widetilde{\tau}^{-1}(\widetilde{P}))v_{\widetilde{\tau}^{-1}(\widetilde{P})}}
{\Vert(I + \widetilde{P} - \widetilde{\tau}^{-1}(\widetilde{P}))v_{\widetilde{\tau}^{-1}(\widetilde{P})}\Vert} \right) = 
\zeta_{\widetilde{\tau}^{-1}(\widetilde{P})}\left(\frac{(I + \widetilde{\sigma}(\widetilde{P}) - 
\sigma\widetilde{\tau}^{-1}(\widetilde{P}))w_{\widetilde{\tau}^{-1}(\widetilde{P})}}
{\Vert (I + \widetilde{\sigma}(\widetilde{P}) - \sigma\widetilde{\tau}^{-1}(\widetilde{P}))w_{\widetilde{\tau}^{-1}(\widetilde{P})}\Vert} \right)
\]
for each $\widetilde{P}$ in $\widetilde{\mathcal{P}}$.  
Lemma \ref{close} implies that
 $\widetilde{U}$ maps $\ran \widetilde{P}$ to
$\ran \widetilde{\sigma}(\widetilde{P})$ for each $\widetilde{P}$ in $\widetilde{\mathcal{P}}$, and so
$\widetilde{U}\widetilde{P}\widetilde{U}^* = \widetilde{\sigma}(\widetilde{\mathcal{P}})$.  Thus
$\bigl((\widetilde{\mathcal{P}}, \widetilde{\mathcal{Q}}, \widetilde{\sigma}), \widetilde{U}\bigr)$ is in $p^{-1}(\mathcal{O})$,
and we define
\[
\psi\left((\widetilde{\mathcal{P}}, \widetilde{\mathcal{Q}}, \widetilde{\sigma}), 
\bigoplus_{P\in\mathcal{P}}\zeta_P\right) = 
\bigl((\widetilde{\mathcal{P}}, \widetilde{\mathcal{Q}}, \widetilde{\sigma}), \widetilde{U}\bigr).
\]
As with $\phi$, Lemma \ref{technical}(ii) implies that $\psi$ is continuous.  The maps $\psi$ and $\phi$ are inverses of one another
and thus $\phi$ is a homeomorphism.
 \end{proof}
 
We remark that because $B_n$ is a metric space, it is paracompact \cite[Theorem 41.4]{MK2} and Hausdorff \cite[Section 21]{MK2}, and thus $p: E_n \longrightarrow B_n$ is a
 fibration for each natural number $n$ \cite[Corollary 2.7.14]{Spanier}; we will need this fact in Chapter 6. 
 
\section{Unitary equivalence of  normal matrices}\label{S: equiv}

We now return to our study of matrices.  Let $X$ be a topological space.  Recall that $C(X)$ is the $\mathbb{C}$-algebra of complex-valued continuous 
functions on $X$ and that $M_n(C(X))$ is the ring of $n$-by-$n$ matrices with entries in $C(X)$. For $A\in M_n(C(X))$, we define the adjoint of $A$ pointwise, 
and $A$ is defined to be \emph{normal} if $AA^*=A^*A$. The matrix $A$ is \emph{multiplicity free} if, for each $x\in X$, the eigenvalue of $A(x)$ are distinct. 

Suppose $A$ and $B$ in $M_n(C(X))$ are normal, multiplicity-free, and have the same characteristic polynomial. Then for each $x$ in $X$, the matrices 
$A(x)$ and $B(x)$ have the same distinct eigenvalues.  This set of eigenvalues does not come with a natural ordering.  However, given an eigenvalue
$\lambda$ of $A(x)$, we can associate to $\lambda$ the spectral projection $P(x)_\lambda$ of $A(x)$; that is, the orthogonal projection of 
$\mathbb{C}^n$ onto the $\lambda$-eigenspace of $A(x)$.  Similarly, we can associate to $\lambda$ the spectral projection $Q(x)_\lambda$ of $B(x)$. 
We thus have a bijection from the set $\mathcal{P}$ of spectral projections of $A(x)$ to the set $\mathcal{Q}$ of spectral projections of $B(x)$. 
This determines an element of $B_n$.  The spectral projections of $A(x)$ and $B(x)$ vary continuously as functions of $x$, and therefore 
we can assign to the pair $(A, B)$ a continuous map $\Phi_{A,B}: X \longrightarrow B_n$.  

\begin{proposition}\label{P: U equiv}
Matrices $A$ and $B$ in $M_n(C(X))$ that are normal, multiplicity-free, and have the same characteristic polynomial are unitarily equivalent
if and only if $\Phi_{A,B}: X \longrightarrow B_n$ lifts to a continuous map $\widetilde{\Phi}_{A,B}: X \longrightarrow E_n$.
\end{proposition}

\begin{proof}
If $UAU^* = B$ for some $U$ in $U_n(C(X))$, then, by basic linear algebra,  for each $x$ in $X$, the unitary matrix $U(x)$ conjugates each spectral projection
of $A(x)$ to the corresponding spectral projection of $B(x)$; that is, we have $U(x)P(x)_\lambda U(x)^* = Q(x)_\lambda$ for all $x$ and 
$\lambda$.  Therefore $(\Phi_{A,B}(x), U(x))$ is an element of $E_n$ for each $x$ in $X$, and we can define 
$\widetilde{\Phi}_{A,B}: X \longrightarrow E_n$ by $\widetilde{\Phi}_{A,B}(x)=(\Phi_{A,B}(x), U(x))$. This is continuous because the assignments
$x\mapsto \Phi_{A,B}(x)$ and $x\mapsto U(x)$ are continuous by definition.

Conversely, suppose that $p\widetilde{\Phi}_{A,B} = \Phi_{A,B}$ for some continuous map $\widetilde{\Phi}_{A,B}: X \longrightarrow E_n$.  For each $x$ in $X$, 
write $\widetilde{\Phi}_{A,B}(x) = (\Phi_{A,B}(x),U(x))$.  For each eigenvalue $\lambda$ of $A(x)$ and $B(x)$, we have $U(x)P(x)_\lambda U(x)^* = Q(x)_\lambda$
by the definitions of $\Phi_{A,B}$ and $E_n$, and thus $U(x)A(x)U(x)^* = B(x)$ for each $x$ in $X$.
The assignment $x \longmapsto U(x)$ is a continuous map from $X$ to $U_n$ that defines an element $U$ in $U_n(C(X))$, and $UAU^* = B$. 
\end{proof}

\subsection{Cohomology with local coefficients}
Proposition \ref{P: U equiv} tells us that to approach the question of whether $A$ is unitarily equivalent to $B$, we need to know when the map $\Phi_{A,B}$ can be
lifted to the bundle $E_n$.  In order to do this, we will employ obstruction theory, which utilizes \emph{cohomology with local coefficients}.  We sketch the basic 
ideas of cohomology with local coefficients here and refer the interested reader to \cite[Chapter 5]{DK}, \cite[Section 3.H]{Ha}, or \cite[Chapter VI]{Wh} for
more information. In fact, there are two equivalent approaches, both of which will be useful for us. To describe the first, let $\Gamma$ be a group and suppose
we have a representation $\rho$ of $\Gamma$ on an abelian group $A$; i.e.,
a group homomorphism $\rho: \Gamma \longrightarrow \Aut(A)$.  Then $A$ is a left $\mathbb{Z}\Gamma$-module via the action
\[
\left(\sum_{g\in\Gamma} m_g g\right)\cdot a = \sum_{g\in\Gamma} m_g\rho(g)(a);
\]
we often write $A$ as $A_\rho$ to highlight the dependence of the module action on the choice of $\rho$.
Now suppose $X$ is a connected\footnote{The assumption that $X$ be connected is not essential; if $X$ has multiple connected components, each component
can be treated individually. Alternatively, though more technically advanced, one could replace fundamental groups in this discussion with fundamental groupoids.} 
topological space with universal cover $\widetilde{X}$ and basepoint $x_0$. Let $\Gamma=\pi_1(X,x_0)$, and let $S_*(\widetilde{X})$ denote the integral singular
chain complex over $\widetilde{X}$. The groups  $S_*(\widetilde{X})$ are modules over $\mathbb{Z}\Gamma$ by the action of the covering transformations. 
The \emph{cohomology $H^*(X; A_\rho)$ of $X$ with local coefficients in $A$} is the cohomology of the cochain complex 
$\Hom_{\mathbb{Z}\Gamma}(S_*(\widetilde{X}), A)$. If the representation $\rho$ is trivial, then $H^*(X; A_\rho)$ is just $H^*(X;A)$, 
the ordinary cohomology of $X$ with coefficients in the abelian group $A$.   

Equivalently, representations $\pi_1(X,x_0) \longrightarrow \Aut(A)$ correspond to isomorphism classes of bundles over $X$ with fiber $A$; see
\cite[Theorems VI.1.11 and VI.1.12]{Wh}. If $\Pi$ is such a bundle of groups over $X$ corresponding to $A_\rho$, then $H^*(X; \Pi)\cong H^*(X; A_\rho)$ 
can be described via cochains whose values on singular simplices correspond to lifts of the singular simplices  to $\Pi$. See 
\cite[Section 3.H]{Ha} for more details. Yet another approach, utilized in \cite[Section VI.2]{Wh}, is to think of a singular cochain as assigning to a singular chain
$\sigma:\Delta^k \longrightarrow X$ a value in the fiber over $\sigma(v_0)$, where $v_0$ is the initial vertex of $\Delta^k$. Of course, this is equivalent 
to prescribing a lift of all of $\sigma$, as $\Pi$ is a covering space of $X$. With some more effort, suitable versions of cellular cohomology with systems
of local coefficients can be defined; see \cite[Section VI.4]{Wh}. 

Now, suppose we have a fibration $p:E\longrightarrow X$ with fibers $F_x$ over $x\in X$.  Furthermore, assume that the  $F_x$ are $k$-simple, which means
that the action of  $\pi_1(F_x)$ on $\pi_k(F_x)$ is trivial. This $k$-simplicity implies that there are canonical isomorphisms $\pi_k(F_x,f_{x,0})\cong \pi_k(F_x,f_{x,1})$
for any two basepoints $f_{x,0},f_{x,1}\in F_x$. In fact, we obtain bijections $\pi_k(F_x,f_{x,0}) \longrightarrow [S^k,F_x]$, the set of free homotopy classes of
maps from $S^k$ to $F_x$ \cite[Corollary 6.60]{DK}, so we don't have to worry about basepoints in the fibers at all. As a consequence, the fibration 
$p:E \longrightarrow X$ yields a bundle of groups $\pi_k(\mc F)$ over $X$ with fibers $[S^1,F_x]\cong \pi_1(F_x)$; see \cite[Proposition 6.62]{DK} or
\cite[Example VI.1.4]{Wh}. Bundles of groups arising in this way also possess nice topological descriptions when considered as groups with representations 
of $\pi_1(X,x_0)$: Let $F_0$ denote the fiber over the basepoint $x_0\in X$, and consider $\pi_k(F_0)\cong [S^k, F_0]$.  If we have an element of $\pi_k(F_0)$ 
represented by a map $h_0:S^k \longrightarrow F_0$, then the homotopy lifting property of fibrations implies that  a loop $\gamma$ in $X$ determines
(uniquely up to homotopies) an extension of $h_0$ to $H: S^k \times I \longrightarrow E$ over $\gamma$. If $h_0 = H|_{S^k\times \{0\}}$, 
then $H|_{S^k \times \{1\}}$ determines a new map $h_1= H|_{S^k\times \{1\}}:S^k \longrightarrow F_{0}$. So this lifting process determines a map
$\rho:\pi_1(X,x_0) \longrightarrow \Aut(\pi_k(F_{0}))$ by $\gamma \mapsto ([h_0]\to[h_1])$. If we denote $\pi_k(F_{{0}})$ with this action of $\pi_1(X,x_0)$ by
$\pi_k(F_{{0}})_\rho$, the categorical equivalence between bundles of groups over $X$ and groups possessing $\pi_1(X,x_0)$ actions identifies $\pi_k(\mc F)$
with $\pi_k(F_{{0}})_\rho$. The reader should consult \cite{DK} or \cite{Wh} for further details.

\subsection{Back to matrices}
Now, returning to matrices, let $\Phi_{A,B}: X\longrightarrow B_n$ be as above for two normal multiplicity-free matrices in $M_n(C(X))$ with the same
characteristic polynomial, and let $\Phi_{A,B}^*E_n$ be the pullback of $E_n$. Because the fibers of $E_n$ are homeomorphic to the torus $T^n$, so are 
the fibers $F_x$ of $\Phi_{A,B}^*E_n$ over $X$, and $\pi_1(F_x)\cong \Z^n$. As $\Z^n$ is abelian, the group $\pi_1(F_x)$ acts trivially on itself by conjugation
(see \cite[Exercise 114]{DK}), so  $F_x$ is $1$-simple.
Therefore, we can form the bundle of groups $\pi_1(F_x)$, and we will denote this bundle of groups by $\Pi_{A,B}$. 

\begin{theorem}\label{T: U equiv}
Let $X$ be a connected CW complex, and suppose $A$ and $B$ are normal multiplicity-free matrices
in $M_n(C(X))$ that have the same characteristic polynomial.  Then there exists a  unique cohomology class 
$[\theta(A,B)] \in H^2(X; \Pi_{A,B})$ such that $A$ and $B$ are unitarily equivalent if and only if $[\theta(A,B)]=0$. 
\end{theorem}

\begin{proof}
The proof is by obstruction theory. We recall the relevant theorem\footnote{Our particular statement  is a hybrid of the phrasings and notations in 
\cite{DK} and \cite{Wh}.}; see \cite[Theorem 7.37]{DK} and \cite[Corollary 5.7]{Wh}: Given a CW complex $X$, a fibration $p:E \longrightarrow Y$ 
with fiber $F$, and a map $f:X \longrightarrow Y$, suppose that $\widetilde{f}^k : X^k \longrightarrow E$ is a lift of $f$ over the $k$-skeleton  $X^k$ of $X$.  
Further, suppose that $F$ is $k$-simple.  Let $\pi_k(\mc F)$ denote the $\pi_k(F)$ bundle associated to $f^*E$ over $X$. 
Then there is an \emph{obstruction class} $[\theta^{k+1}(\widetilde{f}^k)]$ in the cohomology group
$H^{k+1}(X;\pi_k(\mc F))$ such that  $[\theta^{k+1}(\widetilde{f}^k)] = 0$ if and only if the restriction
$\widetilde{f}^k|_{X^{k-1}}$ can be extended to a lifting of $f$ over $X^{k+1}$.  

In our situation, the fiber $F$ is homeomorphic to $T^n$, so $\pi_k(F)$ is trivial unless $k = 1$, in which case
$\pi_1(F) \cong \Z^n$. Thus $F$ is trivially $k$-simple for $k \neq 1$. 
For $k = 1$, we obtain the bundle of groups $\Pi_{A,B}$ over $X$, as described above. 

Now consider $\Phi_{A,B}: X \longrightarrow B_n$. We can 
construct a lift $\widetilde{\Phi}_{A,B}^0 : X^0 \longrightarrow E_n$ by just choosing a point
$\bigl((\mathcal{P}, \mathcal{Q}, \sigma), U\bigr)$ in $p^{-1}(\Phi_{A,B}(x))$
for each $x$ in $X^0$.
Since $\pi_0(F)$ is trivial, the obstruction theorem ensures that there is a continuous map
$\widetilde{\Phi}_{A,B}^1 : X^1 \longrightarrow E_n$ lifting $\Phi_{A,B}$ over the $1$-skeleton $X^1$ of $X$. Now we
encounter an obstruction $[\theta^{2}(\widetilde{\Phi}_{A,B}^1)]$ in $H^{2}(X; \Pi_{A,B})$. 
The obstruction theorem says that this class vanishes if and only if 
$\widetilde{\Phi}_{A,B}^1|_{X^0}$ extends to a lift $\widetilde{\Phi}_{A,B}^2: X^2 \longrightarrow E_n$. If 
$[\theta^{2}(\widetilde{\Phi}_{A,B}^1)] = 0$, then such a $\widetilde{\Phi}_{A,B}^2$ exists.  Furthermore, 
because $\pi_k(F)$ vanishes for $k > 1$, there are no other
obstructions to lifting $\Phi_{A,B}$ on all of $X$ to obtain a map $\widetilde{\Phi}_{A,B}: X \longrightarrow E_n$.

Our construction of the obstruction $[\theta^{2}(\widetilde{\Phi}_{A,B}^1)]$ ostensibly depends on our choices of
$\widetilde{\Phi}_{A,B}^0$ and $\widetilde{\Phi}_{A,B}^1$. 
First, let $\widetilde{\Phi}_{A,B}^0$ and $\widehat{\Phi}_{A,B}^0$ be two lifts of $\Phi_{A,B}$ over the $0$-skeleton. 
These lifts are \emph{vertically (or fiber-wise) homotopic} (see \cite[page 291]{Wh}), because any two lifts of a vertex of $X^0$
lie in the same fiber over $B_n$ and so can be connected by a path in that fiber, which is homeomorphic to $T^n$ and hence is path connected. 
Second, let $\widetilde{\Phi}_{A,B}^1$ and $\widehat{\Phi}_{A,B}^1$ denote the lifts of $\widetilde{\Phi}_{A,B}^0$ and $\widehat{\Phi}_{A,B}^0$ 
on $X^1$ guaranteed by the obstruction theorem.
By the same argument that we just used above, the restrictions $\widetilde{\Phi}_{A,B}^1|_{X^0}$ and 
$\widehat{\Phi}_{A,B}^1|_{X^0}$ are vertically homotopic. This puts us in the setting of  \cite[Theorem VI.5.6.3]{Wh}, which implies that 
$\theta^{2}(\widetilde{\Phi}_{A,B}^1)$ and $\theta^{2}(\widehat{\Phi}_{A,B}^1)$ are cohomologous. 
Thus the obstruction cohomology class in $H^{2}(X; \Pi_{A,B})$ is independent of our choices 
in the construction. Denoting this class by $[\theta(A,B)]$, we have shown that 
$\Phi_{A,B}$ possesses a lifting if and only if $[\theta(A, B)] = 0$. 
Thus by Proposition \ref{P: U equiv}, the matrices $A$ and $B$ are unitarily equivalent if and only if $[\theta(A,B)] = 0$.
\end{proof} 

An immediate corollary is a strengthening of Grove and Pedersen's \cite[Theorem 1.4]{GP}, which implies that if $X$ is a $2$-connected compact 
CW complex then any multiplicity-free normal $A$ in $M_n(C(X))$ can be diagonalized. 

\begin{corollary}\label{C: GP cor}
If $X$ is a simply-connected (not necessarily compact) CW complex and $\Hom(H_2(X),\Z)=0$ (in particular if $H_2(X)$ is torsion), 
then any two normal multiplicity-free matrices $A$ and $B$ in $M_n(C(X))$ with the same eigenvalues at each point are unitarily equivalent. 
In particular, any normal multiplicity-free matrix in $M_n(C(X))$ is diagonalizable.
\end{corollary}

\begin{proof}
Because $X$ is simply connected, we see that $\Pi_{A,B}$ is the trivial $\Z^n$ bundle and so $[\theta(A,B)] \in H^2(X;\Z^n)$. By the universal coefficient theorem
\cite[Theorem 53.1]{MK}, we have $H^2(X;\Z^n)\cong \Hom(H_2(X),\Z^n) \oplus \Ext(H_1(X);\Z^n)$. The supposition that $X$ is simply connected
implies that $H_1(X) = 0$ and thus $\Hom(H_2(X),\Z^n)\cong \displaystyle\oplus_{i=1}^n\Hom(H_2(X),\Z)$. So, given the assumption that 
$\Hom(H_2(X),\Z)=0$, the obstruction class
 $[\theta(A,B)]$ vanishes, and the unitary equivalence follows from Theorem \ref{T: U equiv}. 

To show that any normal multiplicity-free matrix $A$ in $M_n(C(X))$ is diagonalizable, it follows from Goren and Lin \cite[Theorem 1.6]{GL} that the 
simple connectivity of $X$ implies that the characteristic polynomial $\mu$ of $A$ splits as $\prod_{i=1}^n(\lambda- d_i(x))$ for some collection 
$d_1, d_2, \dots, d_n$ of complex-valued continuous functions on $X$.  Let $D \in M_n(C(X))$ be the diagonal matrix with $d_i$ in the $i$th diagonal slot.
By the preceding paragraph, $A$ is unitarily equivalent to $D$. 
\end{proof}
 
\begin{example}\label{E: circle}
Let us re-examine an example from \cite{GP}.
Let $X=S^1$, and let $A$ be the normal matrix
\[
A(z)=\begin{pmatrix}
0&z\\
1&0
\end{pmatrix}.
\]

The characteristic polynomial of $A$ is 
\[
\mu(z, \lambda)=\lambda^2-z,
\]
which is multiplicity free but does not globally split (i.e., it does not factor over $C(X)$). Therefore, by \cite{GP}, $A$ cannot be diagonalized.

What about the unitary equivalence class of $A$? As $S^1$ can be treated as a cell complex with no cells of dimension greater than $1$, 
we see that $H^2(S^1; \Pi_{A,B})=0$ for any normal matrix $B$ with the same characteristic polynomial $\mu$. Therefore $A$ and $B$ are unitarily equivalent
if $B$ is any such matrix.  In other words, there is only one unitary equivalence class of matrices with characteristic polynomial $\mu(z, \lambda)=\lambda^2-z$.
\end{example}

\section{Naturality and the extension to non-CW spaces}\label{S: naturality}

In this section, we show that the obstructions $[\theta(A,B)]$ of Theorem \ref{T: U equiv} are natural with respect to maps in an appropriate sense. We will begin by 
considering cellular maps of CW complexes, but the techniques will allow us to generalize both Theorem \ref{T: U equiv} and our naturality statements to certain
non-CW spaces. For convenience, we will often assume that spaces carrying matrices are pointed (i.e. that they come equipped with basepoints) and that maps
and homotopies preserve the basepoints. In these instances, the spaces $B_n$ and $E_n$ are not assumed to have basepoints, and $\Phi_{A,B}$ is never
a pointed map. First, we recall some background material. 

\subsection{Some more homotopy theory}
Let us briefly recall from \cite[Section VI.2]{Wh} the appropriate categorical framework for maps of cohomology with local coefficients. In \cite{Wh}, 
Whitehead defines a category $\ms L^*$ whose objects are triples $(X,A;\mc G)$ with $(X,A)$ being a space pair (in the category of compactly generated spaces,
 which includes all locally compact Hausdorff spaces \cite[I.4.1]{Wh} and so all CW complexes \cite[II.1.6.1]{Wh}) and $\mc G$ being a system of local 
 coefficients (bundle of groups) over $X$. A morphism $\phi:(X,A;\mc G) \longrightarrow (Y,B;\mc H)$ is then a continuous map of spaces
 $\phi_1:(X,A)\longrightarrow (Y,B)$ along with a bundle homomorphism $\phi_2: \phi_1^*\mc H \longrightarrow \mc G$. Here, if $H$ is the fiber group of $\mc H$
and\footnote{If any of the spaces in our discussion are disconnected, then these statements should be modified either to a collection of statements over different
 connected components or, more direct but also a bit more fancy, a statement in terms of fundamental groupoids. We leave these modifications for the reader. 
 See \cite[Section VI.1]{Wh}.} $\rho_{\mc H}: \pi_1(Y) \longrightarrow  \Aut(H)$ is the monodromy that determines $\mc H$, then $\phi_1^*\mc H$ is the system
 of local coefficients whose fiber group is $H$ and whose monodromy is determined by the composition $\pi_1(X)\xr{\phi_{1*}} \pi_1(Y) \xr{\rho_{\mc H}} \Aut(H)$.
In this setting, we obtain cohomology maps $\phi^*:H^*(Y,B;\mc H) \longrightarrow H^*(X,A;\mc G)$. In our situation, given a map $f:(X,A) \longrightarrow (Y,B)$ 
and a system of local coefficients $\mc H$ over $Y$, we will always take $\mc G = f^*\mc H$, so our $\phi_2$ will always be the identity  
and we simply write $f^*:H^*(Y,B;\mc H) \longrightarrow H^*(X,A;f^*\mc H)$. 

We should also say a few words about homotopies.  For basepoint-preserving homotopies from $X$ to $Y$, it is useful to replace the usual $X\times I$ by  the
``reduced prism'' $X\wedge I_+$, which is  homeomorphic to $X\times I/\{x_0\}\times I$. This space has a natural basepoint --- the image of $\{x_0\}\times I$
in the quotient --- and so serves as a good domain for basepoint-preserving homotopies. See \cite[Section III.2]{Wh}. We will denote the basepoint $[x_0]$. 
If $X$ is a CW complex then so is $X \wedge I_+$ by \cite[Example II.1.5]{Wh}. Whitehead considers the action of homotopic maps on cohomology groups 
in \cite[Section VI.2]{Wh} using the standard prism $X\times I$, but the arguments easily adapt to the reduced prism. Given a system of local coefficients
$\mc G$ on $X$, the prism $X\wedge I_+$ is given the system $p^*\mc G$, where $p:X\wedge I_+ \longrightarrow X$ is the projection. Then one defines a homotopy
between $\phi, \psi: (X,A;\mc G) \longrightarrow (Y,B;\mc H)$ via a map $\eta: (X\wedge I_+, A\wedge I_+;p^*\mc G) \longrightarrow (Y,B;\mc H)$, and we get 
$\phi^*=\psi^*: H^*(Y,B;\mc H) \longrightarrow H^*(X,A;\mc G)$ by \cite[VI.2.6*]{Wh}. In our case, given a homotopy $h:(X\wedge I_+, [x_0]) \longrightarrow (Y,y_0)$ 
between $f:(X,x_0) \longrightarrow (Y,y_0)$ and $g:(X,x_0) \longrightarrow (Y,y_0)$,  rather than work with something of the form $p^*\mc G$, we would prefer to 
work with $h^*\mc H$ on $X\wedge I_{+}$, which restricts to  $f^*\mc H$ and $g^*\mc H$ on $X\times \{0\}$ and $X\times \{1\}$. However, it is not difficult to 
observe that $f^*\mc H\cong g^*\mc H$ and that $h^*\mc H\cong p^*f^*\mc H\cong p^*g^*\mc H$; this frees us to utilize $h^*\mc H$ without violating Whitehead's
framework. For this, it is useful to turn to the viewpoint of bundles of groups as groups with $\pi_1$ actions.  We first observe that the two compositions
$\pi_1(X,x_0)\xr{f_*} \pi_1(Y,y_0) \xr{\rho_{\mc H}} \Aut(H)$ and 
$\pi_1(X,x_0)\xr{g_*} \pi_1(Y,y_0) \xr{\rho_{\mc H}} \Aut(H)$ are identical, because $f$ and $g$ are basepoint preserving homotopic maps. Similarly, 
the compositions 
\begin{align*}
\pi_1(X\wedge I_+,[x_0])\xr{h_*} \pi_1(Y,y_0) \xr{\rho_{\mc H}} \Aut(H)\\
\pi_1(X\wedge I_+,[x_0])\xr{(fp)_*} \pi_1(Y,y_0) \xr{\rho_{\mc H}} \Aut(H)\\
\pi_1(X\wedge I_+,[x_0])\xr{(gp)_*} \pi_1(Y,y_0) \xr{\rho_{\mc H}} \Aut(H)
\end{align*}
are all identical because $fp\sim h\sim gp$. So, in this case, it makes sense to say that 
$f^*=g^*: H^*(Y,B;\mc H) \longrightarrow H^*(X,A;f^*\mc H)= H^*(X,A;g^*\mc H)$. 
The equality is really an abuse of notation; we should replace it with a canonical isomorphism. However, in what follows we will repeat this abuse rather
than overburden the notation. 

\subsection{Back to matrices} We can now return to our study of obstructions to unitary equivalence of matrices. 

\begin{definition}
Suppose $f:Y \longrightarrow X$ is a map of spaces and that $A\in M_n(C(X))$. We define the \emph{pullback of $A$}, denoted $f^*A$, to be the 
matrix in $M_n(C(Y))$ such that $(f^*A)(y)=A(f(y))$. 
\end{definition}

Notice that if $A$ in $M_n(C(X))$ is normal and multiplicity free, then so is $f^*A$, as these are pointwise determined properties. Similarly, if $A$ and $B$
in $M_n(C(X))$ have the same characteristic polynomial, then so do $f^*A$ and $f^*B$, and if $U\in M_n(C(X))$ is unitary, so is $f^*U$. 

\begin{proposition}\label{P: naturality}
Let $f:Y \longrightarrow X$ be a cellular map of CW complexes, and let $A,B\in M_n(C(X))$ be multiplicity-free normal matrices with the same characteristic polynomial. 
Let $[\theta(A,B)]\in H^2(X;\Pi_{A,B})$ be as in Theorem \ref{T: U equiv}. Then $[\theta(f^*A,f^*B)] = f^*[\theta(A,B)]$ in $H^2(Y; f^*\Pi_{A,B})$. 
\end{proposition}

\begin{proof}
We first notice that $\Phi_{f^*A,f^*B}: Y \longrightarrow E_n$ is equal to the composition $Y\xr{f}X\xr{\Phi_{A,B}} B_n$. If $f$ is cellular, then the obstruction to
lifting the composition is exactly $f^*[\theta(A,B)]$ by  basic properties of obstruction theory that follow directly from the definitions \cite[Theorem V.5.3]{Wh}.
\end{proof}

\begin{example}\label{E: natural example}
Proposition \ref{P: naturality} can yield some results that are \emph{a priori} unexpected if the subject is approached from a purely analytic point of view. For example, 
suppose $(X,Z)$ is any CW pair and that $A,B\in M_n(C(X))$ are multiplicity-free normal matrices. If the restrictions of $A$ and $B$ to $Z$ are not unitarily equivalent,
then certainly $A$ and $B$ cannot be unitarily equivalent over all of $X$. However, the proposition shows that in some cases there will be a surprising converse to this.
In particular, let $i:Z \longrightarrow X$ be the inclusion and suppose that the restriction $i^*:H^2(X;\Pi_{A,B})\longrightarrow H^2(Z; i^*\Pi_{A,B})$ is injective. 
Proposition \ref{P: naturality} implies that $[\theta(i^*A,i^*B)] = i^*[\theta(A,B)]$, so if $[\theta(A,B)]\neq 0$, then $[\theta(i^*A,i^*B)]\neq 0$. 

Here's a concrete example: Consider $S^1\times S^2$, and let $i: S^2\into S^1\times S^2$ take $S^2$ to some $\{x_0\}\times S^2$. Then 
$i^*:H^2(S^1\times S^2;\Z^n)\longrightarrow H^2(S^2; \Z^n)$ is an isomorphism. So if two multiplicity free normal matrices with the same characteristic polynomial
and no monodromy of roots are not unitarily equivalent over $S^1\times S^2$, it follows that their restrictions to $S^2$ cannot be unitarily equivalent. In fact, clearly, 
none of the restrictions to any $\{x\}\times S^2$ can be unitarily equivalent, as any such inclusion can be made cellular. This leads also to the interesting conclusion 
that if $A$ and $B$ are two multiplicity free normal matrices with the same characteristic polynomial over $S^2$, then any extensions of $A$ and $B$ with the same characteristic polynomial over $S^1\times S^2$ must be unitarily equivalent. 
\end{example}

Next, we need a corollary to Proposition \ref{P: naturality} that will serve as a useful lemma later in this section.

\begin{corollary}\label{C: pullback}
Let $(X,x_0)$ be a pointed CW complex, let $(Z,z_0)$ be an arbitrary pointed space, and let $f,g:(X,x_0)\longrightarrow (Z,z_0)$ be homotopic maps.
Suppose $A$ and $B$ in $M_n(C(Z))$ are normal and multiplicity free with a common characteristic polynomial. Then we have $[\theta(f^*A,f^*B)] = [\theta(g^*A,g^*B)]$ in 
$H^2(X;f^*\Pi_{A,B}) = H^2(X;g^* \Pi_{A,B})$. 
\end{corollary}

\begin{proof}
Let $h:X \wedge I_+\longrightarrow Z$ be the (basepoint-preserving) homotopy from $f$ to $g$, and, for $s=0,1$, let $i_s:X\longrightarrow X\times \{s\}$ be the inclusions. 
Then $hi_0=f$ and $hi_1=g$, and $i_0,i_1$ are cellular maps. 

By Theorem \ref{T: U equiv}, the class $[\theta(h^*A,h^*B)]$ in $H^2(X\times I;h^*\Pi_{A,B})$ is  a well-defined  obstruction  to  $h^*A$ and $h^*B$ being unitarily equivalent. 
  By Proposition \ref{P: naturality} and the definitions,  
\[
[\theta(f^*A,f^*B)]=[\theta(i_0^*h^*A,i_0^*h^*B)]=i_0^*[\theta(h^*A,h^*B)]\in H^2(Y; i_{0}^*h^*\Pi_{A,B})
\]
and
\[
[\theta(g^*A,g^*B)]=[\theta(i_1^*h^*A,i_1^*h^*B)]=i_1^*[\theta(h^*A,h^*B)]\in H^2(Y; i_{1}^*h^*\Pi_{A,B}).
\]
But $i_0$ and $i_1$ are obviously (basepoint-preserving) homotopic maps, so $i_0^*[\theta(h^*A,h^*B)]=i_1^*[\theta(h^*A,h^*B)]$. The corollary follows. 
\end{proof}

Using the preceding results, we can now define an obstruction to the unitary equivalence of two normal multiplicity-free matrices on any space $Z$ that is homotopy
equivalent to a CW complex: Suppose $(Z,z_0)$ is a pointed locally compact Hausdorff space, and suppose $(X,x_0)$ is a CW pair that is (basepoint-preserving)
homotopy equivalent to $(Z,z_0)$. Let $f:(Z,z_0)\longrightarrow (X,x_0)$ and $g:(X,x_0)\longrightarrow (Z,z_0)$ be homotopy inverses to one another. 
Suppose that $A$ and $B$ in $M_n(C(Z))$ are normal and multiplicity free. Then we have the obstruction $[\theta(g^*A,g^*B)]$ in $H^2(X;g^*\Pi_{A,B})$, where
 $\rho$ is the map $\pi_1(Z,z_0)\longrightarrow \Aut(\Z^n)$ obtained by composing the induced map $(\Phi_{A,B})_*: \pi_1(Z, z_0) \longrightarrow \pi_1(B)$ 
  and the representation $\pi_1(B^n)\longrightarrow \Aut(\Z^n)$ determined by the bundle $E_n\longrightarrow B_n$.
 
 \begin{definition}
 \emph{Define} $[\theta(A,B)] \in H^2(Z;f^*g^*\Pi_{A,B})=H^2(Z;\Pi_{A,B})$ to be $[\theta(A,B)]=f^*[\theta(g^*A,g^*B)]$. 
 \end{definition}

\begin{remark}\label{R: CW agrees}
Note that if $Z$ is itself a CW complex, then this definition agrees with our previous usage by taking both $f$ and $g$ to be the identity map $Z\longrightarrow Z$.
\end{remark}

\begin{proposition}\label{P: nonCW ob}
Suppose $(Z,z_0)$ is a locally compact Hausdorff space that is (basepoint-preserving) homotopy equivalent to a CW pair $(X,x_0)$. Let $A,B\in M_n(C(Z))$
be normal and multiplicity free.  The class $[\theta(A,B)]$ is independent of the choice of homotopy equivalence used to define it, and it vanishes if and only if 
$A$ and $B$ are unitarily equivalent.
\end{proposition}

\begin{proof}
Suppose that $(\widehat X,\hat x_0)$ is another CW pair that is (basepoint-preserving) homotopy equivalent to $(Z,z_0)$ by homotopy inverses  
$\hat f:(Z,z_0)\longrightarrow (\widehat X,\hat x_0)$ and $\hat g:(\widehat X,\hat x_0) \longrightarrow (Z,z_0)$. Let $k$ be a cellular approximation
to $\hat f g$ by a basepoint-preserving homotopy; see \cite[Theorem 4.8]{Ha}). Then $\hat gk\sim \hat g\hat fg\sim g$ in the following diagram:

\begin{diagram}
 (Z,z_0)   &\pile{\rTo^f \\\lTo_g } &  (X,x_0)\\
	\uTo^{\hat g}\dTo_{\hat f}&\ldTo_k\\
	(\widehat X,\hat x_0).
\end{diagram}

Now, we can perform the following computation:

\begin{align*}
f^*[\theta(g^*A,g^*B)]&=f^*[\theta(k^*\hat g^*A,k^*\hat g^*B)]&\text{by Corollary \ref{C: pullback}}\\
&=f^*k^*[\theta(\hat g^*A,\hat g^*B)]&\text{by Proposition \ref{P: naturality}}\\
&=f^*g^*\hat f^*[\theta(\hat g^*A,\hat g^*B)]&\text{pullbacks by homotopic maps}\\
&=\hat f^*[\theta(\hat g^*A,\hat g^*B)]&\text{pullbacks by homotopic maps.}
\end{align*}

This shows that our definition of $[\theta(A,B)]$ on $Z$ is independent of choices. 

For the second claim, first suppose that $A$ and $B$ are unitarily equivalent. Then $B=UAU^*$, and $g^*B=(g^*U)(g^*A)(g^*U^*)=(g^*U)(g^*A)(g^*U)^*$. 
So $g^*B$ is unitarily equivalent to $g^*A$ and $[\theta(A,B)] = f^*[\theta(g^*A,g^*\hat g^*B)] = f^*(0) = 0$. 

Next, suppose that $[\theta(A,B)] = f^*[\theta(g^*A,g^*B)] = 0$. Then we have that $g^*[\theta(A,B)] = g^*f^*[\theta(g^*A,g^*B)] = 0$. But $fg$ is homotopic to the identity, 
so $[\theta(g^*A,g^*B)] = 0$, which implies by Theorem \ref{T: U equiv} that $g^*A$ and $g^*B$ are unitarily equivalent. Pulling back by $f$ a unitary matrix that
realizes the unitary equivalence of $g^*A$ and $g^*B$, as in the argument of the preceding paragraph, shows that $f^*g^*A$ and $f^*g^*B$ are unitarily equivalent. 
By Proposition \ref{P: U equiv}, this means that $\Phi_{f^*g^*A,f^*g^*B}:Z\longrightarrow B_n$ lifts to $E_n$. Unraveling the definitions, we see that
$\Phi_{f^*g^*A,f^*g^*B} = g\circ f\circ \Phi_{A,B}$, which is homotopic to $\Phi_{A,B}$. As $g\circ f\circ \Phi_{A,B}$ has a lift to $E_n$, so does $\Phi_{A,B}$,
by the homotopy lifting extension property of fibrations. Therefore, again by Proposition \ref{P: U equiv}, the matrices $A$ and $B$ are unitarily equivalent. 
\end{proof}

Lastly, now that we have defined an obstruction for non-CW spaces, we can show that it is also natural. 

\begin{proposition}\label{P: nonCW naturality}
Let $h:(Z,z_0)\longrightarrow (\widehat Z,\hat z_0)$ be a map of locally-compact Hausdorff spaces that are (basepoint-preserving) homotopy equivalent to CW complexes. 
Let $A$ and $B$ in $M_n(C(\widehat Z))$ be normal and multiplicity free. Then $[\theta(h^*A,h^*B)] = h^*[\theta(A,B)]$. 
\end{proposition}

\begin{proof}
Suppose we have maps $f:(Z,z_0)\longrightarrow (X,x_0)$ and $\hat f:(\widehat Z,\hat z_0) \longrightarrow (\widehat X,\hat x_0)$ 
that are (basepoint-preserving) homotopy 
equivalences to CW pairs with inverses $g:(X,x_0) \longrightarrow (Z,z_0)$ and $\hat g:(\widehat X,\hat x_0) \longrightarrow (\widehat Z,\hat z_0)$.
Consider the following diagram, in which $k$ is a cellular approximation to $\hat fhg$. We have 
$\hat gk \sim \hat g\hat fhg \sim  hg$ and $kf\sim \hat fhgf\sim \hat f h$. 

\begin{diagram}
(Z,z_0)& \pile{\rTo^f \\\lTo_g } &  (X,x_0)\\
\dTo^h &&   \dTo^{k} \\
(\widehat Z,\hat z_0)& \pile{\rTo^{\hat f} \\\lTo_{\hat g} } &  (\widehat X,\hat x_0).\\
\end{diagram}

Now we compute

\begin{align*}
[\theta(h^*A,h^*B)]&=f^*[\theta(g^*h^*A,g^*h^*B)]&\text{definition}\\
&=f^*[\theta(k^*\hat g^* A,k^* \hat g^*B)]&\text{by Corollary \ref{C: pullback}}\\
&=f^*k^*[\theta(\hat g^* A, \hat g^*B)]&\text{by Proposition \ref{P: naturality}}\\
&=h^*\hat f^*[\theta(\hat g^* A, \hat g^*B)]&\text{pullback by homotopic maps}\\
&=h^*[\theta(A,B)]&\text{definition}.
\end{align*}
\end{proof}

\begin{remark}
In particular, if $(Z,z_0)$ and $(\widehat Z,\hat z_0)$ in the statement of Proposition \ref{P: nonCW naturality} are CW pairs but $h$ is not necessarily a cellular map, 
then Proposition \ref{P: nonCW naturality} extends Proposition \ref{P: naturality} to this setting; see also Remark \ref{R: CW agrees}. 
\end{remark}

\section{Monodromy}\label{S: monodromy}

So far, our invariants $[\theta(A,B)]$ have lived in the groups $H^2(X;\Pi_{A,B})$, where $\Pi_{A,B}$ is a bundle of groups over $X$ having fiber $\Z^n$. 
In this section, we will show that, up to isomorphism, our $\Z^n$ bundles depend only on the common characteristic polynomial of $A$ and $B$ and not on
the matrices themselves. For this, it will be convenient in this section to return to thinking of a bundle of groups as a group over the basepoint $x_0$ of $X$
together with a  $\pi_1(X,x_0)$ action. In our case, this corresponds to a representation $\rho:\pi_1(X,x_0)\longrightarrow \Z^n$. 

Let $x_0 \in X$ be a fixed basepoint,  let $A \in M_n(C(X))$ be normal and multiplicity-free, and let  
$\Lambda=\{\lambda_1,\ldots,\lambda_n\}$ be the eigenvalues of $A(x_0)$, listed in some arbitrary order. 
If $\gamma$ is a loop in $X$ based at $x_0$, then $\gamma$ induces a permutation of $\Lambda$ that depends only on the class of 
$\gamma$ in $\pi_1(X) = \pi_1(X,x_0)$. Details can be found in \cite{GL}. The basic idea is that if we choose an eigenvalue 
$\lambda$ of $A(x_0)$ and then follow the continuously varying eigenvalue as we move around the loop $\gamma$, then,
when we return to $x_0$, we may arrive back at a different eigenvalue. Altogether, this yields a monodromy assignment 
from  the homotopy class $[\gamma]$ to  $S_\Lambda$, the permutation group on $\Lambda$. In fact, following all the eigenvalues 
as we move around the loop leads to a one-parameter family of configurations of $n$ distinct points in $\C$, and so one obtains a representation 
$\pi_1(X)\longrightarrow \mc B_n$, where $\mc B_n$ is the braid group on $n$ strands. Our monodromy action on $\Lambda$ then corresponds to the
 map $\mc B_n\longrightarrow S_\Lambda$ determined by how the braid permutes the endpoints. 
Similarly, as we move along $\gamma$ we also obtain a 1-parameter family of collections of $n$ linearly independent 
eigenspaces which will be mutually orthogonal if $A$ is normal. Corresponding to the monodromy permutation of eigenvalues is the corresponding 
permutation of eigenspaces (interpreted as a bijection of sets whose elements are subspaces of $\C^n$, not in terms of specific linear maps).
 Similarly, we have permutations of spectral projections.

\begin{proposition}\label{clean}
Let $\mu$ be the common characteristic polynomial of normal multiplicity-free matrices $A$ and $B$ in $M_n(C(X))$, let 
$\mf m_\mu: \pi_1(X,x_0) \longrightarrow S_\Lambda$ be the representation determined by the monodromy
of the zeros of $\mu$ around loops, and, for $\alpha \in S_\Lambda$, let $\Sigma_{\alpha}$ denote the corresponding permutation matrix.  
Then the representation $\rho:\pi_1(X) \longrightarrow \Aut (\Z^n)$ corresponding to the bundle of groups $\Pi_{A,B}$ takes $[\gamma]$ to 
$\Sigma_{\mf m_\mu([\gamma])}$. 
In particular, $\rho$ depends only on the polynomial $\mu$. 
\end{proposition}

\begin{proof}
Choose a basepoint $x_0$ in $X$, and let $\gamma$ be a loop in $X$ based at $x_0$. By definition, the representation $\rho([\gamma])$ is determined 
by the action of the loop $\Phi_{A,B}\circ \gamma$ on $\pi_1(F_0)$, where $\pi_1(F_0)$ is the fundamental group of the fiber $F_0$ of $E_n$ over
$\Phi_{A,B}(x_0)$.  From Proposition \ref{fiber structure}, we know that $F_0$ can be viewed as $\prod_{P\in \mathcal{P}(0)} S^1$, where
 $\mathcal{P}(0)$ is the collection of spectral projections of $A(x_0)$, and hence $\pi_1(F_0) \cong \prod_{P\in \mathcal{P}(0)}\pi_1(S^1) \cong \Z^n$. 

More precisely, let\footnote{It would be more consistent to write $\sigma(0)$, but this choice will make the notation a bit easier below.} 
$\bigl((\mathcal{P}(0), \mathcal{Q}(0), \sigma_0), U(0)\bigr)$ be an arbitrary point in the fiber $F_0$, and let $P_1(0), P_2(0),\ldots, P_n(0)$ be the 
elements of $\mc P(0)$ written in the order determined by the ordering of the eigenvalues in $\Lambda$.  By Proposition \ref{fiber structure}, every
element of $F_0$ has a unique form 
\[
\widetilde{U} = \sum_{j=1}^n z_{j}\sigma_0(P_j(0))U(0)P_j(0),
\]
as each parameter  ${z}_j$ runs over $S^1$. Collectively, this gives the homeomorphism $T^n\cong F_0$. Consequently, via this identification, we can
describe the $i$th generator $[\ell_{i}]\in \pi_1(F_0)$ by the loop
\[
\ell_i(z)= z\sigma_0(P_i(0))U(0)P_i(0) + \sum_{j\neq i}\sigma_0(P_j(0))U(0)P_j(0)
\]
for $z\in S^1$ with its standard orientation. 

Now, as recalled in our review of cohomology with local coefficients in Section \ref{S: equiv}, the action of $\pi_1(X)$ on $[\ell_i]$ will be represented by any
loop ``at the other end'' of a lift of $S^1\times I$ over $\Phi_{A,B}\circ \gamma$ that extends $\ell_i$. We will construct such a lift explicitly. First, we
parameterize the loop $\Phi_{A,B}\circ\gamma$ by $t\in I$.  Note that the spectral projections of $A(\gamma(t))$ vary continuously with $t$ and are distinct
at every point, so, given our choice of ordering $\mc P(0)=\{P_j(0)\}$, the path $\gamma$ determines paths of spectral projections $\{P_j(t)\}$ that agree with
our  $\{P_j(0)\}$ at $t=0$ (explaining our earlier choice of notation). Because $\gamma$ is a loop,  we have that $\mc P(1)=\mc P(0)$, but in general $P_j(1)$ 
is not necessarily equal to $P_j(0)$.  In fact, if $\lambda_j$ is the eigenvalue of $A(\gamma(0))=A(x_0)$ corresponding to the projection  $P_j(0)$, then 
$P_j(1)$ is precisely the projection corresponding to eigenvalue $\mf m_{\mu}(\gamma)(\lambda_j)$; moving along $\gamma$ permutes the spectral 
projections exactly as it permutes the corresponding eigenvalues. 

Next,  let $\eta$ be a lift of $\Phi_{A,B}\circ \gamma$ to $E_n$ such that $\eta(0)=\bigl((\mathcal{P}_0, \mathcal{Q}_0, \sigma_0), U(0)\bigr)$. We can
write $\eta(t)= \bigl((\mathcal{P}_0(t), \mathcal{Q}_0(t), \sigma_t), U(t) \bigr)$, with each $P_j(t)\in \mc P(t)$. Now parameterize $S^1\times I$ by
coordinates $(z,t)$, and define
\[
H(z,t)=z\sigma_t(P_i(t))U(t)P_i(t) + \sum_{j\neq i}\sigma_t(P_j(t))U(t)P_j(t).
\]
Proposition \ref{fiber structure} guarantees that this is a lift of $\Phi_{A,B}\circ\gamma$, and we have clear agreement with $\ell_i$ at $t=0$. At $t=1$,
we have the loop 
\[
z\mapsto z\sigma_1(P_i(1))U(t)P_i(1) + \sum_{j\neq i}\sigma_1(P_j(1))U(t)P_j(1),
\]
which is evidently the generator of $\pi_1(F_0)$ corresponding to the spectral projection associated to the eigenvalue $\mf m_{\mu}(\gamma)(\lambda_i)$.

Therefore, we see that the action of $\gamma$ on the generators of $\pi_1(F)\cong \Z^n$ is precisely as claimed.
\end{proof}

\begin{corollary} Suppose the only homomorphism from $\pi_1(X)$ to $\mc B_n$ is the trivial one.
Then $\theta(A,B)$ is in $H^2(X;\Z^n)$.
\end{corollary}

\begin{proof}
By \cite[Theorem 1.4]{GL}, if the only homomorphism $\pi_1(X)\longrightarrow \mc B_n$ is trivial, then any polynomial with coefficients in $C(X)$ and leading coefficient $1$ splits as $\prod_{i=1}^n(\lambda- d_i(x))$; 
in particular, by Proposition \ref{clean}, the monodromy of roots is trivial. Thus $\rho$ is trivial, and the claim follows.
\end{proof}

\begin{corollary} Suppose the only homomorphism from $\pi_1(X)$ to $\mc B_n$ is the trivial one, and
also suppose that $H^2(X;\Z) = 0$. Then any two multiplicity-free normal 
matrices $A$ and $B$ in $M_n(C(X))$ with the same characteristic polynomial are unitarily equivalent.
\end{corollary}

\begin{proof}
The preceding corollary implies that $\theta(A,B)$ is in $H^2(X;\Z^n)$. But $H^2(X;\Z^n) \cong (H^2(X;\Z))^n$.  Now apply 
Theorem 3.2.
\end{proof}

\section{Obstruction relations}\label{S: OR}

In this section, we will consider how the invariants $[\theta(A,B)]$ are related to each other as the matrices $A$ and $B$ vary. In previous sections
our main consideration was whether or not $[\theta(A,B)]=0$. Now we will be more concerned with particular elements of cohomology groups, and, in order
for us to be precise, it will be necessary for us to look under the hood a bit more and pin down better descriptions of our cohomology groups and obstruction elements. 

\subsection{Review of the obstruction cochain}
First, let us describe in more detail the definition of the obstruction cochain $\theta^{2}(\td \Phi^1_{A,B})$ as used in the proof of Theorem \ref{T: U equiv}. 
More generally, recall (\cite[Section VI.5]{Wh}) that if $f:X\longrightarrow B$ is a map from a CW complex $X$ to a space $B$, if $p:E\longrightarrow B$ is a fibration, 
and if $\td f^k:X^k\longrightarrow E$ is a lift of the restriction of $f$ to the $k$-skeleton $X^k$, then we have defined an obstruction cochain $\theta^{k+1}(\td f^k)$. 
This cellular cochain is defined as follows: First, we may as well assume $X$ is connected, or we can work on each component separately. Because $X$ is connected, 
we can assume that $X$ has a single $0$-cell to serve as a basepoint and that every cell attachment map is a basepoint-preserving map. Let $\e^{k+1}$ be a cell
of $X$, with characteristic map $h:(\Delta^{k+1}, \bd \Delta^{k+1})\longrightarrow (X^{k+1}, X^k)$. The composition of $\td f^k$ with the restriction of $h$ to 
$\bd \Delta^{k+1}$ gives a lift map $\bd \Delta^{k+1}\longrightarrow E$ or, equivalently, to the pullback of $E$ over $\Delta^{k+1}$. As $\Delta$ is contractible, 
the pullback of $E$ over $\Delta$ is a trivial fibration (up to a homotopy equivalence that we can assume fixes the fiber over the basepoint) and so is homotopy
equivalent to the fiber $F_0$ of $E$ over the basepoint. So our lift of $\bd \Delta^{k+1}$ to the pullback of $E$ over $\Delta^{k+1}$ defines an element of 
$[S^k,F_0]$, the set of homotopy classes of maps of $k$-spheres to $F_0$. Given the assumption that $F_0$ is $k$-simple, we can identify  $[S^k,F_0]$ with
$\pi_k(F_0)$ without concern about basepoints. This assignment from cells of $X$ to elements of $\pi_k(F_0)$ gives a cochain
$\theta^{k+1}(\td f^k)\in C^{k+1}(X; \pi_k(\mc F))$, where $\pi_k(\mc F)$ denotes the local system of coefficients on $X$ with fiber $\pi_k(F_0)$ determined
by the bundle $f^*E$. As noted in Section \ref{S: equiv}, the results of \cite[Sections VI.5 and VI.6]{Wh} imply that $\theta^{k+1}(\td f^k)$ is a cocycle, that its
cohomology class $[\theta^{k+1}(\td f^k)]$ depends only on $\td f^{k-1}$,  and that $[\theta^{k+1}(\td f^k)]=0$ if and only if $\td f^{k-1}$ can be extended to a lift 
of $f$ over $X^{k+1}$.  It is useful to observe  that  finding a lift of $f:X\longrightarrow B$ to $E$ is equivalent to finding a section of the induced bundle 
$f^*E$ over $X$ (see \cite[Section VI.5]{Wh}), and, in fact, the definition of $\theta^{k+1}(\td f^k)$ remains identical viewing the problem in this light.  

\subsection{Basing the coefficient systems}
Let us return now to our obstructions $[\theta(A,B)]$ in $H^2(X;\Pi_{A,B})$, where $A,B\in M_n(C(X))$ are normal multiplicity-free matrices with a common 
characteristic polynomial $\mu$. Here $\Pi_{A,B}$ is the bundle of groups over $X$ with fibers $\pi_1(F_x)$, where $F_x\cong T^n$ is the fiber of 
$\Phi_{A,B}^*E_n$ over $x\in X$. By the results of Section \ref{S: monodromy}, we know that the bundle structure of $\Pi_{A,B}$ depends only on the common
characteristic polynomial of $A$ and $B$. In particular, Proposition \ref{clean} says that if we choose an ordering $\Lambda$ of the common eigenvalues of 
$A$ and $B$ over the basepoint $x_0\in X$, then, up to isomorphism, $\Pi_{A,B}$ is the bundle corresponding to the representation 
$\rho:\pi_1(X,x_0)\longrightarrow \Aut(\pi_1(F_0))\cong \Aut(\Z^n)$ determined by the permutation of the roots of the characteristic polynomial as we move 
along a loop. Technically, in the language of Proposition \ref{clean}, we have $\rho([\gamma])= \Sigma_{\mf m_\mu([\gamma])}$, where $\Sigma$ is the
permutation matrix corresponding to the permutation $\mf m_\mu([\gamma])\in S_{\Lambda}$. 

The nice thing about $\Z^n_\rho$ is that it does not refer to $A$ and $B$ at all, except through their common characteristic polynomial, and so it provides a
neutral coefficient system in which to compare elements of $H^2(X;\Pi_{A,B})$ for various $A$ and $B$. However, in order to do this, we need to be explicit 
about our isomorphisms $\Z^n_\rho\cong \Pi_{A,B}$. Already this is a bit of notational abuse, as $\Z^n_\rho$ and $\Pi_{A,B}$ live in different categories:
$\Z^n_\rho$ is a group with a  $\pi_1(X,x_0)$ representation and  $\Pi_{A,B}$ is a bundle of groups. To remedy this, \cite[Theorem VI.1.12]{Wh} tells us how
to construct a specific bundle of groups corresponding to $\Z^n_\rho$  with fiber $\Z^n$ identically over the basepoint, and we can abuse notation by allowing
$\Z^n_\rho$ also to stand for this bundle. As we already know that $\Z^n_\rho$ and $\Pi_{A,B}$ are isomorphic (discrete) bundles, it suffices to specify an
isomorphism between them over $x_0$ in order to determine an isomorphism completely. We will refer to this as ``basing'' $\Pi_{A,B}$ because we can think of such
an isomorphism as determining a basis of $\pi_1(F_{0})$ by imposing the image of the standard basis of $\Z^n$. This is analogous to orienting a manifold $M^n$
via an isomorphism from the constant bundle with $\Z$ coefficients (and an arbitrary fixed generator of $\Z$) to the orientation bundle with fibers $H_n(M,M-\{x\})$. 
As in that setting, the exact basing, which is determined completely by our ordering of the eigenvalues over $x_0$,  will not necessarily be so important as the 
establishment of a single reference frame by which to compare objects.

 If $x_0$ is the basepoint of $X$, then the fiber of $\Pi_{A,B}$ over $x_0$ has the form $\pi_1(F_{0})$, where $F_{0}=\{((\mc P_0,\mc Q_0,\sigma_0),U)\}$ with 
 $(\mc P_0,\mc Q_0,\sigma_0)=\Phi_{A,B}(x_0)$ and $U$ ranging over the set of unitary matrices taking the eigenspaces of $A$ to the corresponding eigenspaces
of $B$.  We choose the standard  basis $\{b_i\}_{i=1}^n$ for $\Z^n$, and we suppose that we have chosen an ordering $\Lambda=\{\lambda_1,\ldots,\lambda_n\}$ 
of the roots of $\mu(x_0)$. This determines corresponding orderings of the spectral projections of $A(x_0)$ and $B(x_0)$. Now,  we can define an isomorphism 
$\mf o_{A,B}:\Z^n\longrightarrow \pi_1(F_{0})$ such that $\mf o_{A,B}(b_i)=[\ell_i]$, where $[\ell_i]\in \pi_1(F_{0})\cong [S_1,F_{0}]$ is defined as in the proof of
Proposition \ref{clean}. Note that the definition there of the loop $\ell_i$ depended on a choice of matrix $U_0$ to obtain a basepoint $((\mc P,\mc Q,\sigma),U_0)$ 
in the fiber, but the free homotopy class $[\ell_i]\in [S_1,F_{0}]$ does not depend on this choice.   Because we know that $\Z^n_\rho$ and $\Pi_{A,B}$ are abstractly
isomorphic, the map $\mf o_{A,B}$ extends to an isomorphism of systems of local coefficients.

\subsection{The transposition relation}
We will now utilize our bundle isomorphisms $\mf o_{A,B}$ to study the relationship between $[\theta(A,B)]$ and $[\theta(B,A)]$.

Observe that the space $E_n$ possesses an involution $\td \nu: E_n\longrightarrow E_n$ given by
$$\td \nu((\mc P,\mc Q, \sigma),U)= ((\mc Q,\mc P,\sigma^{-1}), U^{-1}).$$
The map $\td \nu$ is not a bundle map; it does not preserve fibers of $B_n$. However, it covers the involution $\nu$ of $B_n$ given by 
\[
\nu(\mc P,\mc Q, \sigma)= (\mc Q,\mc P,\sigma^{-1}),
\]
so we have a commutative diagram
\begin{diagram}
E_n&\rTo^{\td \nu} & E_n\\
\dTo^p&&\dTo^p\\
B^n&\rTo^\nu &B_n.
\end{diagram}
Furthermore, we can see from the definitions that $\nu\Phi_{A,B}=\Phi_{B,A}$, so  $\td \nu$ induces a bundle map
$\td \nu_\#: \phi_{A,B}^*E_n\longrightarrow \phi_{B,A}^*E_n$ and hence a map of local systems of coefficients that we will denote 
$\td\nu_*: \Pi_{A,B}\longrightarrow \Pi_{B,A}$. 

\begin{lemma}\label{L: orient ABBA}
The following diagram commutes:

\begin{diagram}
\Z^n_{\rho}&\rTo^{\mf o_{A,B}}&\Pi_{A,B}\\
\dTo^{-1}&&\dTo^{\td \nu_*}\\
\Z^n_{\rho}&\rTo^{\mf o_{B,A}}&\Pi_{B,A}.
\end{diagram}
\end{lemma}

\begin{proof}
Let $F_0$ continue to denote the fiber of $\Phi_{A,B}^*E_n$ over $x_0\in X$, and let $F_0'$ denote the fiber of $\Phi_{B,A}^*E_n$ over $x_0$. By definition, 
the isomorphism $\mf o_{A,B}$ takes the generator $b_i$ of $\Z^n$ to $[\ell_i]$, where the loop $\ell_i$ in $F_0$ has 
\[
\ell_i(z)= z\sigma_0(P_i)U_0P_i + \sum_{j\neq i}\sigma_0 (P_j)U_0P_j
\]
as its unitary coordinate; see the proof of Proposition \ref{clean} and note that we are free to simplify notation a bit here because we will not be lifting a cylinder
as we did in that proof. Here, we have $\{P_i\}=\mc P_0$, though with our chosen ordering. 

From the definition, the map $\td \nu$ takes the loop $\ell_i$ in $F_0$ to a loop $\td \nu \ell_i$ in $F_0'$ that has 
$\td \nu\ell_i(z)=(\ell_i(z))^{-1}$ in its unitary coordinate.  We claim that 
\[
(\ell_i(z))^{-1}= z^{-1}\sigma_0^{-1}(Q_i)U^{-1}_0Q_i + \sum_{j\neq i}\sigma_0^{-1}(Q_j)U^{-1}_0Q_j.
\]
To see this, we consider the products $\sigma^{-1}(Q_j)U^{-1}_0Q_j\sigma_0 (P_k)U_0P_k.$ First, observe that $\sigma_0 (P_k)=Q_k$ and $\sigma^{-1}(Q_j)=P_j$, 
so we can simplify this expression to $P_jU^{-1}_0Q_jQ_k U_0P_k.$ If $j\neq k$, then $Q_jQ_k=0$ as composition of two projections in orthogonal directions. 
If $j=k$, then $Q_jQ_k=Q_jQ_j=Q_j$. Furthermore, as $U_0$ takes the range of $P_k$ to the range of $Q_k$ by definition of $E_n$, we actually have
$Q_jU_0P_k=U_0P_k$. So 
\[
P_jU^{-1}_0Q_jQ_j U_0P_j=P_jU^{-1}_0 U_0P_j=P_jP_j=P_j.
\]
Therefore, multiplying $\ell_i(z)$ by our claimed inverse, distributing, and removing terms that equal zero, we obtain the expression $\sum_j P_j$; this is the
identity because the $P_j$ are $n$ mutually orthogonal projections whose ranges span $\mathbb{C}^n$.

Now, suppose  $\mf o_{B,A}(b_i)=[\ell'_i]$, where $\ell'_i$ is defined analogously to $\ell_i$. For convenience, we can use $U_0^{-1}$ as our basepoint in 
$F'_0$, though, again, the choice of basepoint doesn't really matter. Then we see that $\mf o_{B,A}$ takes $b_i$ to the class of the loop 
\[
z\sigma^{-1}_0(Q_i)U^{-1}_0Q_i + \sum_{j\neq i}\sigma^{-1}_0 (Q_j)U_0^{-1}Q_j.
\]
But this is the negative of the class of the loop $\td \nu\ell_i(z)$, proving the lemma. 
\end{proof}

Next, let us relate $[\theta(A,B)]\in H^2(X;\Pi_{A,B})$ with $[\theta(B,A)]\in H^2(X;\Pi_{B,A})$. For this, we utilize that a map of local systems of coefficients
induces a (covariant) homomorphism on cohomology.

\begin{lemma}\label{L: transpose}
The map $\td \nu_*: H^2(X;\Pi_{A,B}) \longrightarrow H^2(X;\Pi_{B,A})$ takes $[\theta(A,B)]$ to $[\theta(B,A)]$. 
\end{lemma}

\begin{proof}
Let $\theta^2(\td \Phi_{A,B}^1)$ denote the obstruction cochain determined by the lift $\td \Phi_{A,B}^1:X^1\longrightarrow E_n$ of the restriction of $\Phi_{A,B}$ 
to $X^1$. As we reviewed at the beginning of this section, $\theta^2(\Phi_{A,B}^1)$ acts on a cell $\Delta^2$ by thinking of $\td \Phi_{A,B}^1$ as providing a section
of the pullback of $E_n$ to $\Delta^2$, which determines a loop $\td \Phi_{A,B}^1:\bd\Delta^2\longrightarrow F_0$, after identifying the pullback over $\Delta^2$ 
as $\Delta^2\times F_0$, up to a fiberwise homotopy equivalence (fixing $F_0$). Composing this section over $\bd \Delta^2$ with the pullback of $\td \nu$ to 
$\Delta^2$ then yields an element of $\pi_1(F_0')$ which is precisely the value of the obstruction cochain $\theta^2(\td \nu\td \Phi_{A,B}^1)$. But 
$\td \nu\td \Phi_{A,B}^1$ is a lift over $X^1$ of $\nu\Phi_{A,B}=\Phi_{B,A}$, so we can define $\td \Phi^1_{B,A}=\nu \td \Phi_{A,B}^1$. Also, taking the image of 
a loop in $F_0$ to a loop in $F_0'$ via $\td \nu$ is precisely $\td \nu_*$, so we obtain  
\[
\theta^2(\td \Phi_{B,A}^1)=\theta^2(\td \nu\td \Phi_{A,B}^1)=\td \nu_*\theta^2(\td \Phi_{A,B}^1).
\]
But these $\theta^2$ are the cochains that represent the obstruction cohomology classes, so we have 
\[
[\theta(B,A)]=\td \nu_*[\theta^2(A,B)].
\]
\end{proof}

\begin{remark}\label{R: ABBA}
Informally, we would really like to say something like $[\theta(B,A)]=-[\theta(A,B)]$, which makes some intuitive sense. However, part of the point of the preceding
discussion is that such a statement does not quite make sense because $[\theta(A,B)]$ and $[\theta(B,A)]$ live in groups that have \emph{isomorphic} coefficient 
systems but not \emph{identical} coefficient systems. That said, Lemma \ref{L: orient ABBA}, together with Lemma \ref{L: transpose}, shows that if we base the 
coefficient systems $\Pi_{A,B}$ and $\Pi_{B,A}$ via $\mf o_{A,B}$ and $\mf o_{B,A}$ and then pull back both $[\theta(A,B)]$  and
$[\theta(B,A)]$ to $H^2(X;\Z^n_\rho)$ using these bases, then the images of $[\theta(A,B)]$  and $[\theta(B,A)]$ in $H^2(X;\Z^n_\rho)$ are negatives of each other. 
\end{remark}

\subsection{The additivity relation}
Suppose that $A,B,C\in M_n(C(X))$ are normal and multiplicity free with a common characteristic polynomial.  
We study the relationship between the obstructions $[\theta(A,B)]$, $[\theta(B,C)]$, and $[\theta(A,C)]$.

For this, we first construct a bundle morphism
\[
m_{A,B,C}: \Phi^*_{A,B}E_n \oplus \Phi^*_{B,C}E_n \longrightarrow \Phi^*_{A,C}E_n.
\]
Over a point $x\in X$, the fiber $\Phi^*_{A,B}E_n$ consists of elements of the form $\bigl((\mathcal{P}, \mathcal{Q}, \sigma), U\bigr)$, where 
$(\mathcal{P}, \mathcal{Q}, \sigma)=\Phi_{A,B}(x)$. Similarly,  the fiber of $\Phi^*_{B,C}E_n $ at $x$ consist of elements of the form 
$\bigl((\mathcal{Q}, \mathcal{R}, \tau), V\bigr)$. Then we define $m_{A,B,C}$ over $x$ by
\[
m_{A,B,C,x}\Bigl(\bigl((\mathcal{P}, \mathcal{Q}, \sigma), U\bigr), \bigl((\mathcal{Q}, \mathcal{R}, \tau), V\bigr)\Bigr)
= \bigl((\mathcal{P}, \mathcal{R}, \tau\sigma), VU\bigr).
\]

This is well defined because if $\Phi_{A,B}(x)=(\mathcal{P}, \mathcal{Q}, \sigma)$ and $\Phi_{B,C}(x)=(\mathcal{Q}, \mathcal{R}, \tau)$, then $\Phi_{A,C}(x)$ 
must be $(\mathcal{P}, \mathcal{R}, \tau\sigma)$, as we see by considering the eigenspaces of $A(x)$, $B(x)$, and $C(x)$. Furthermore, if $U$ takes the 
eigenspaces of $A(x)$ to the corresponding eigenspaces of $B(x)$ and if $V$ takes the eigenspaces of $B(x)$ to the corresponding eigenspaces of $C(x)$, 
then $VU$ must take the eigenspaces of $A(x)$ to the corresponding eigenspaces of $C(x)$. 
As $x$ ranges over $X$, the maps $m_{A,B,C,x}$ induces a map of coefficient systems 
$m_{A,B,C\#}: \Pi_{A,B} \oplus \Pi_{B,C} \longrightarrow \Pi_{A,C}$. 

\begin{lemma}\label{L: sums}
We have a commutative diagram
\begin{diagram}
\Z^n_\rho\oplus \Z^n_\rho & \rTo^{\mf o_{A,B}\oplus \mf o_{B,C}}&\Pi_{A,B}\oplus \Pi_{B,C}\\
\dTo^+&&\dTo^{m_{A,B,C\#}}\\
\Z^n_\rho& \rTo^{\mf o_{A,C}} &\Pi_{A,C}.
\end{diagram}
Here $+$ denotes the addition operation in $\Z^n_\rho$. 
\end{lemma}

\begin{proof}
Let us verify the commutativity over the basepoint $x_0$. This suffices, as all maps are bundle maps. We can assume we have fixed an ordering $\Lambda$ of
the zeros of $\mu(x_0)$.  For convenience, we can also choose a basepoint $U_0$ in the fiber $F_0$ of $\Phi^*_{A,B}E_n$ over $x_0$ and a basepoint $V_0$ 
in the fiber $F'_0$ of $\Phi^*_{B,C}E_n$ over $x_0$. We let $V_0U_0$ be a basepoint in the fiber $F''_0$ of $\Phi^*_{B,C}E_n$ over $x_0$.

Let $[\ell_i]$ be the generators of $\pi_1(F_0)$ employed in Proposition \ref{clean} and earlier in this section, i.e. 
\[
\ell_i(z)= z\sigma_0(P_i)U_0P_i + \sum_{j\neq i}\sigma_0 (P_j)U_0P_j.
\]
Similarly, let 
\[
\ell'_i(z)= z\tau_0(Q_i)V_0Q_i + \sum_{j\neq i}\tau_0 (Q_j)V_0Q_j
\]
be loops generating $\pi_1(F'_0)$.

Next, consider the products of the form $\tau_0 (Q_i)V_0Q_i\sigma_0 (P_k)U_0P_k$. Because $\sigma_0(P_k)=Q_k$, this becomes 
$\tau_0 (Q_i)V_0Q_iQ_k U_0P_k$. If $i\neq k$, then $Q_iQ_k=0$, as these are orthogonal projections; in this case, the entire product is $0$. If $i=k$, then
we have $Q_iQ_k U_0P_k=Q_kQ_kU_0P_k=Q_kU_0P_k=U_0P_k$,  because $U_0$ takes the range of $P_k$ to the range of $Q_k$. Therefore
\[
\tau_0 (Q_k)V_0Q_k\sigma_0 (P_k)U_0P_k=\tau_0 (Q_k)V_0U_0P_k=\tau_0\sigma_0 (P_k)V_0U_0P_k.
\]
Multiplying and distributing, we see that if $j\neq k$, then 
\[
m_{A,B,C, x_0\#}([\ell_i]\oplus [\ell'_k])= z\tau_0\sigma_0 (P_i)V_0U_0P_i+z\tau_0\sigma_0 (P_k)V_0U_0P_k+ \sum_{j\neq i,k}\tau_0\sigma_0 (P_j)V_0U_0P_j,
\]
while if $i=k$, we have 
\[
m_{A,B,C, x_0\#}([\ell_i]\oplus [\ell'_k])= z^2\tau_0\sigma_0 (P_i)V_0U_0P_i+ \sum_{j\neq i}\tau_0\sigma_0 (P_j)V_0U_0P_j.
\]

Comparing with the standard representations of generators of $\pi_1(T^n)$, these computations demonstrate the commutativity of the diagram.
\end{proof}

\begin{remark}
It follows that the induced map
\[
m_{A,B,C*}: H^2(X; \Pi_{A,B}) \oplus H^2(X; \Pi_{B,C})\cong H^2(X; \Pi_{A,B} \oplus \Pi_{B,C}) \longrightarrow
H^2(X; \Pi_{A,C})
\]
can be thought of as simple addition in the coefficients, after using our basings to re-identity this product
as a map 
\[
H^2(X; \Z^n_\rho) \oplus H^2(X; \Z^n_\rho)\cong H^2(X; \Z^n_\rho \oplus \Z^n_\rho) \longrightarrow
H^2(X; \Z^n_\rho).
\]
\end{remark}

\begin{lemma}\label{L: sum}
$m_{A,B,C*}\bigl([\theta(A,B)], [\theta(B,C)]\bigr) = [\theta(A,C)]$. 
\end{lemma}

\begin{proof}
We can represent $[\theta(A,B)]$ by $\theta^2(\td f^1)$, where $\td f^1$ is a section of  $\Phi^*_{A,B}E_n$ over $X^1$, and similarly, 
we can represent $[\theta(B,C)]$ by $\theta^2(\td g^1)$, where $\td g^1$ is a section of $\Phi^*_{B,C}E_n$ over $X^1$. As $m_{A,B,C}$ is a bundle map,
the composition
\[
X^1 \xr{\td f^1\oplus \td g^1} \Phi^*_{A,B}E_n \oplus \Phi^*_{B,C}E_n \xr{m_{A,B,C}} \Phi^*_{A,C}E_n,
\]
which we denote $\td h^1$, is a section of  $\Phi^*_{A,C}E_n$ over  $X^1$. Therefore, 
$[\theta(A,C)]=[\theta^2(\td h^1)]$.

On the other hand, by definition, we know that the cochain $\theta^2(\td h^1)$ acts on a $2$-cell $e^2$  of $X$ as follows: the bundle $\Phi^*_{A,C}E_n$ pulls
back to a fiber homotopically trivial $F''_0\cong T^n$ bundle over $\Delta^2$ via the characteristic map $i:(\Delta^2, v_0)\longrightarrow (X,x_0)$, and the section
$\td h^1$ pulls back to a section over $\bd \Delta^2$. Via the fiber homotopy trivialization $i^*\Phi^*_{A,C}E_n\cong \Delta^2\times F''_0$ of the bundle over 
$\Delta^2$, which we can assume is the identity on $F''_0$, and the projection $\Delta^2\times F''_0\longrightarrow F''_0$, we determine a class in  $\pi_1(F''_0)$ 
that is the value of $\theta^2(\td h^1)$ on $e^2$. Of course, $\theta^2(\td f^1)$ and $\theta^2(\td g^1)$ are defined similarly, and 
$m_{A,B,C\#}(\theta^2(\td f^1),\theta^2(\td g^1))$ takes the value on $e^2$ corresponding to the product $m_{A,B,C,x_0*}(\theta^2(\td f^1)(e^2), \theta^2(\td g^1)(e^2))$.
In this last expression, $\theta^2(\td f^1)(e^2)\in \pi_1(F_0)$ and $\theta^2(\td g^1)(e^2)\in \pi_1(F_0')$ are loops and
$m_{A,B,C,x_0*}(\theta^2(\td f^1)(e^2), \theta^2(\td g^1)(e^2))$ is the value under the induced map 
$m_{A,B,C,x_0*}:\pi_1(F_0)\times \pi_1(F_0')\longrightarrow \pi_1(F_0'')$. Up to homotopy, this is simply the product (via $m_{A,B,C}$) of the sections over
$\bd \Delta^2$ of the pullbacks of $\Phi^*_{A,B}E_n$ and $\Phi^*_{B,C}E_n$. But this is precisely the section determined by $\td h^1$. So 
$\theta^2(h^1)=m_{A,B,C*}(\theta(\td f^1), \theta(\td g^1))$. 

Thus we conclude that $[\theta(A,C)]=m_{A,B,C*}\bigl([\theta(A,B)], [\theta(B,C)]\bigr)$. 
\end{proof}

\begin{corollary}\label{C: add corollary}
$m_{A,B,C*}([\theta(A,B)], [\theta(B,C)]) = 0$ if and only if $A$ and $C$ are unitarily equivalent.
\end{corollary}

\begin{proof}
The preceding lemma states that $m_{A,B,C*}([\theta(A,B)], [\theta(B,C)]) = [\theta(A,C)]$, 
and Theorem \ref{T: U equiv} states that $[\theta(A,C)] = 0$ if and only if $A$ and $C$ are unitarily equivalent. 
\end{proof}

Together, Lemmas \ref{L: sum} and \ref{L: sums} basically say that ``$[\theta(A,B)]+[\theta(B,C)] = [\theta(A,C)]$'' once we have chosen basings that allow us to 
normalize all of the elements into the same group $H^2(X;\Z^n_\rho)$ in a consistent way. Corollary \ref{C: add corollary} then says that $A$ and $C$ are 
unitarily equivalent if and only if ``$[\theta(B,C)]=-[\theta(A,B)]$,'' which, using Remark \ref{R: ABBA}, is equivalent to ``$[\theta(B,C)]=[\theta(B,A)]$.'' So two 
matrices $A$ and $C$ are unitarily equivalent if and only if they fail to be unitarily equivalent to a third matrix $B$ via ``the same'' obstruction. In this sense, 
we see that it makes sense to think of our obstructions $[\theta(A,B)]$ as being defined on equivalence classes of matrices and not just on individual matrices.  

Formalizing these observations leads to the following proposition and its corollary. 

\begin{proposition}\label{P: Os}
Let $X$ be a CW complex and $\mu=\mu(x, \lambda)$ a multiplicity free polynomial over $C(X)$. 
Let $A_0\in M_n(C(X))$ be any normal matrix with characteristic polynomial $\mu$.
Let $\mc O_{A_0}$ denote the  set $\{\mf o_{A_0,B}^{-1}([\theta(A_0,B)])\} \subseteq H^2(X;\Z^n_\rho)$ as $B$ runs over all normal matrices in $M_n(C(X))$
with characteristic polynomial $\mu$. Then there is a bijection between $\mc O_{A_0}$ and the set of unitary equivalence classes of normal matrices over $X$
with characteristic polynomial $\mu$. 
\end{proposition}

\begin{proof}
By Lemmas \ref{L: sum} and \ref{L: sums}, 
\[
\mf o_{A_0,C}^{-1}[\theta(A_0,C)]=\mf o_{A_0,B}^{-1}([\theta(A_0,B)])+ \mf o_{B,C}^{-1}[\theta(B,C)]\bigr).
\]

So $\mf o_{A_0,C}^{-1}[\theta(A_0,C)]=\mf o_{A_0,B}^{-1}([\theta(A_0,B)])$ if and only if $\mf o_{B,C}^{-1}[\theta(B,C)]=0$, which in turn is
true if and only if $[\theta(B,C)]=0$, because $\mf o_{B,C}$ is an isomorphism.  So, via Theorem \ref{T: U equiv}, the matrices $B$ and $C$ are 
unitarily equivalent if and only if $\mf o_{A_0,C}^{-1}[\theta(A_0,C)]=\mf o_{A_0,B}^{-1}([\theta(A_0,B)])$, whence the proposition follows. 

\end{proof}

The lemma immediately implies the following remarkable corollary:

\begin{corollary}\label{C: counting}
Given a connected CW complex $X$ and a multiplicity-free  polynomial $\mu=\mu(x, \lambda)$, the number of unitary equivalence classes of normal matrices with 
characteristic polynomial $\mu$ is less than or equal to the cardinality of $H^2(X; \Z^n_\rho)$, where $\rho$ is the representation determined by $\mu$.
In particular, if $H^2(X; \Z^n_\rho)$ is finite, there are a finite number of such equivalence classes, and if $X$ contains a countable number of cells, there are 
a countable number of such equivalence classes\footnote{The countability of unitary equivalence classes is not obvious, given that our matrix components
are $\C$-valued!}. 
\end{corollary}

\begin{example}
It is possible for the inequality implied by the preceding corollary to be strict. For example, if $n=1$, then a multiplicity free normal matrix in $M_1(C(X))$ is just 
a function $f:X\longrightarrow \C$, and, regardless of $H^2(X; \Z_\rho)$, the unitary equivalence class of such a matrix consists of just one element, because
$z^*f(x)z=f(x)$ for any function $z: X\longrightarrow U_1=S^1$. In fact, in this example, $\mu(x)=\lambda-f(x)$, so when $n=1$ there is a bijection between 
elements of $M_1(C(X))$ and characteristic polynomials of such matrices. 

Of course, when $H^2(X; \Z^n_\rho)=0$, for example if $X$ is a point, then equality is realized in the corollary. We will see below that there are less trivial 
examples for which the inequality is strict. 

\end{example}

\subsection{Non-CW spaces} The considerations of this section extend just as well to non-CW spaces, using the techniques of Section 
\ref{S: naturality}. 
Recall that if $Z$ is a locally compact Hausdorff space and that if $f:(Z,z_0)\longrightarrow (X,x_0)$ and $g:(X,x_0)\longrightarrow (Z,z_0)$ are homotopy inverses
to one another, then we defined $[\theta(A,B)]\in H^2(X;\Pi_{A,B})$  as $f^*([\theta(g^*A,g^*B)])$. We can define a basing here by choosing $\hat{\mf o}_{A,B}$
so that the following is a commutative diagram of isomorphisms:

\begin{diagram}
H^2(X; \Z^n_{\rho g_*})&\rTo^{\mf o_{g^*A,g^*B}} &H^2(X;g^*\Phi_{A,B}^*E_n)\\
\dTo^{f^*}&&\dTo^{f^*}\\
H^2(Z; \Z^n_{\rho g_*f_*})=H^2(Z; \Z^n_{\rho} )&\rTo^{\hat{\mf o}_{A,B}}&H^2(Z;f^*g^*\Phi_{A,B}^*E_n)\cong H^2(Z;\Phi_{A,B}^*E_n).
\end{diagram}
Here $\rho$ is the representation of $\pi_1(Z,z_0)$ on $\Phi_{A,B}^*E_n$. 
The invariant to unitary equivalence between the matrices $A$ and $B$ can then be written as either 
$f^*\mf o_{g^*A,g^*B}^{-1}([\theta(g^*A,g^*B)])$ or $\hat{\mf o}_{A,B}^{-1}f^*([\theta(g^*A,g^*B)])$ in 
 $H^2(Z; \Z^n_{\rho})$ , and this vanishes if and only if $A$ is unitary equivalent to $B$.  

Rather than go through the technicalities of translating all the results of this section from $X$ to $Z$, let us  use our existing results to show directly that versions
of Proposition  \ref{P: Os} and Corollary \ref{C: counting} hold for $Z$. Let $A,B,C\in M_n(C(X))$ be normal with the same multiplicity free characteristic 
polynomial. Using both the notation and proof of Proposition \ref{P: Os}, we see that 
$\mf o_{g^*A_0,g^*C}^{-1}[\theta(g^*A_0,g^*C)]=\mf o_{g^*A_0,g^*B}^{-1}([\theta(g^*A_0,g^*B)])$ in $H^2(X; \Z^n_{\rho g_*})$ if and only if 
$\mf o^{-1}_{g^*B,g^*C}[\theta(g^*B,g^*C)]=0$. But this implies that $f^*\mf o_{g^*A_0,g^*C}^{-1}[\theta(g^*A_0,g^*C)]=f^*\mf o_{g^*A_0,g^*B}^{-1}([\theta(g^*A_0,g^*B)])$ in $H^2(X; \Z^n_{\rho })$ if and only if $f^*\mf o_{g^*B,g^*C}^{-1}[\theta(g^*B,g^*C)]=0$ (recall that $f^*$ is an isomorphism
as $f$ is a homotopy equivalence).  In other words, $B$ is unitarily equivalent to $C$ if and only if the obstruction to $A_0$ and $B$ being unitarily 
equivalent is equal to the obstruction to $A_0$ and $C$ being unitarily equivalent.  This is identically the situation that implies Proposition \ref{P: Os}, 
so the analogous conclusions hold over $Z$. A version of Corollary \ref{C: counting} follows.

\section{Relation with Chern classes}

In this section, we make some observations concerning the situation when our characteristic polynomial has a global factorization 
$\mu(x, \lambda)=\prod_{i=1}^n(\lambda-\lambda_i(x))$.  By \cite{GL}, this is equivalent to assuming that the monodromy of the roots of $\mu$ is trivial along all curves. 
In this case, if $A,B\in M_n(C(X))$ are normal and multiplicity free with characteristic polynomial $\mu$, then  $\Pi_{A,B}$ is isomorphic to the trivial $\Z^n$ bundle. 
Moreover, this implies that, for each $i$, the $\lambda_i$ eigenspaces of  $A$ and $B$ determine complex line bundles over $X$. It turns out that, in this setting, 
the obstruction $[\theta(A,B)]$ can be expressed in terms of the Chern classes of the line bundles of maps between these corresponding eigenspace bundles. 

\begin{proposition}\label{P: chern}
Suppose $A$ and $B$ in $M_n(C(X))$ are multiplicity-free normal matrices with a common characteristic polynomial that factors globally over the CW complex $X$. 
Choose eigenvalue functions $\lambda_1, \lambda_2, \dots, \lambda_n$ as described above.  For each $x$ in $X$ and $1 \leq i \leq n$,
let $P_{\lambda_i}(x)$ and $Q_{\lambda_i}(x)$ denote the projections of $\mathbb{C}^n$ onto the $\lambda_i(x)$-eigenspaces of
$A(x)$ and $B(x)$ respectively, and consider the corresponding complex line bundles $\bar P_{\lambda_i}$ and $\bar Q_{\lambda_i}$.  Then 
$[\theta(A, B)] \in H^2(X; \Z^n) = \bigoplus_{i=1}^n H^2(X; \Z)$ is equal to $\bigoplus_{i=1}^n c^1(\Hom(\bar P_{\lambda_i}, \bar Q_{\lambda_i}))$,  
where $c^1(\cdot)$ indicates the first Chern class. 
\end{proposition}

\begin{proof} For each $1 \leq i \leq n$, endow  $\bar P_{\lambda_i}$ and $\bar Q_{\lambda_i}$ with the Hermitian metrics they inherit
as subbundles of the trivial bundle $X \times \mathbb{C}^n$; this induces a Hermitian metric on $\Hom(\bar P_{\lambda_i}, \bar Q_{\lambda_i})$.  By construction, 
$[\theta(A, B)]$ is the obstruction to the existence of a section over $X$ of the torus bundle whose $S^1$ factors at a point $x$ correspond to the set
$\mc U(P_{\lambda_i}(x), Q_{\lambda_i}(x))$ of unitary matrices in
$\Hom(P_{\lambda_i}(x), Q_{\lambda_i}(x))$. Let $\mc U_i$ denote the corresponding $S^1$ bundle over $X$. 
In fact, with our assumptions, we can project each fiber of $\Phi^*_{A,B}E_n$ to the corresponding torus factor $\mc U(P_{\lambda_i}(x), Q_{\lambda_i}(x))=\mc U_{i,x}$,
and this induces a map of bundles of groups $\kappa_i$ from $\Pi_{A,B}$ to the bundle of groups $\pi_1(\mc U_{i,x})$. The maps $\kappa_i$ are projections to direct
summands over each point, and so globally due to the absence of monodromy. So, up to isomorphism, this results in cohomology maps
$\kappa_{i*}:H^2(X;\Pi_{A,B})\longrightarrow  H^2(X;\Z)$, and $\oplus_i \kappa_{i*}$ is an isomorphism  $H^2(X;\Pi_{A,B})\longrightarrow \oplus_i H^2(X;\Z^n)$. 

Now, $[\theta(A,B)]=[\theta^2(\td f^1)]$ is the obstruction to extending a section $\td f^1:X^1\longrightarrow \Phi_{A,B}^*E_n$ to $X^2$, and we see that
$\kappa_{i*}  ([\theta(A,B)])$ will be the obstruction to extending a section over $X^1$ of  the $S^1$ bundle $\mc U_i$.  This obstruction is independent of the particular 
section over $X^1$ by the same arguments employed in the proof of Theorem \ref{T: U equiv}. It only remains to observe that the obstructions to extending to $X^2$ 
sections of circle bundles over $X^1$ is the Chern class of the circle bundle $c^1(\mc U_i)$ (or, equivalently, the Chern class of the equivalent line bundle
$\Hom(P_{\lambda_i}, Q_{\lambda_i})$. But this description of the Chern classes as obstruction classes dates back to Chern's 
original paper, see \cite[Chapter III, Section 1]{Ch}; Chern assumes in this section of his paper that the base space is a complex 
manifold, but this is not essential. See also Steenrod \cite{Ste}, particularly Sections 41.2-41.4.
\end{proof}

\begin{example}\label{E: sphere}
We can now extend another example from \cite{GP}. 
Let $X= \C P^1$, and let $A$ be the normal, multiplicity free matrix
\[
A([z_1, z_2])=\frac{1}{|z_1|^2 + |z_2|^2}\begin{pmatrix}
|z_1|^2 & z_1\bar{z_2}\\
\bar{z_1}z_2 & |z_2|^2
\end{pmatrix}.
\]
The characteristic polynomial is 
\[
\mu([z_1, z_2], \lambda)=\lambda^2-\lambda=\lambda(\lambda-1),
\]
which  globally splits with constant eigenvalue functions $0$ and $1$. In fact, $A$ is the matrix that projects the trivial $\C^2$ bundle over $\C P^1$ to the 
tautological line bundle, which is the $\lambda = 1$ eigenspace bundle of $A$. As this bundle is not trivial, $A$ is not diagonalizable, by the discussion in \cite{GP}. 
Let us see, though, what else we can say about unitary equivalence classes of normal matrices on $\C P^1$ with characteristic polynomial $\mu$. 

If $B$ is any other normal matrix in $M_2(C(\C P^1))$ with characteristic polynomial $\lambda^2-\lambda$, then $B$ will similarly be a projection matrix onto a 
line subbundle of the trivial $\C^2$ bundle. Furthermore, as the polynomial globally splits, we know that any $\Pi_{A,B}$ is isomorphic to the trivial $\Z^n$ bundle
over $\C P^1$. In the discussion that follows, we will tacitly assume that we have utilized our basing procedure from Section \ref{S: OR} to identify all possible 
$H^2(X;\Pi_{A,B})$ with $H^2(X;\Z^2)$.  In this case, the maps  $m_{A,B,C*}$ become simple addition in $H^2(X;\Z^2)$. We can assume we have ordered the
eigenvalues such that  $\lambda_1=1$ and $\lambda_2=0$. 

To pick a more convenient matrix for comparison than the matrix $A$ above, let
\[
D=\begin{pmatrix}1&0\\0&0\end{pmatrix},
\]
which also has characteristic polynomial $\lambda(\lambda - 1)$.  The matrix $D$ projects the trivial $\C^2$ bundle over $\C P^1$ to a trivial $\C$ bundle 
over $\C P^1$ that is also the $\lambda=1$ eigenspace of $D$. The kernel of the projection, corresponding to the $\lambda = 0$ eigenspace bundle, is another
trivial $\C$ bundle. Denote the trivial $\C^n$ bundle by $\epsilon^n$.

Now let $B$ be an arbitrary matrix with characteristic polynomial $\lambda(\lambda - 1)$ and let $E_0$ and $E_1$ be the two eigenspace 
line bundles associated to $B$ with eigenvalues $0$ and $1$, respectively. By Proposition \ref{P: chern}, we see that $[\theta(D,B)] \in H^2(X;\Z^n)$ is equal to 
$c^1(\Hom(\epsilon^1,E_0)) \oplus c^1(\Hom(\epsilon^1, E_1)) = c^1(E_0) \oplus c^1(E_1)$, where  $c^1$ indicates the first Chern class. 
But $E_0 \oplus E_1\cong \epsilon^2$, so $0 = c^1(\epsilon^2) = c^1(E_0\oplus E_1) = c^1(E_0)+c^2(E_1)$. Thus
$[\theta(D,B) ]= c^1(E_1) \oplus -c^1(E_1) \in H^2(\C P^1) \oplus H^2(\C P^1)$. In particular, every obstruction 
$[\theta(D, B)] \in H^2(\C P^1) \oplus H^2(\C P^1)$ must have the form $\alpha\oplus -\alpha$. 

Next, let us show that any element $\alpha \oplus -\alpha \in H^2(\C P^1) \oplus H^2(\C P^1) \cong \Z\oplus \Z$ can be realized by a matrix with characteristic 
polynomial $\lambda(\lambda - 1)$. Every complex line subbundle $L$ of $\epsilon^2$ over $\C P^1$ is determined by a  map $\C P^1\longrightarrow \C P^1$ 
(in the obvious way --- a subbundle  of $\epsilon^2$ consists of a complex line in $\C^2$ over every point of $\C P^1$, which is precisely the information of a map 
$\C P^1\longrightarrow \C P^1$). In particular, the subbundle $L$ is the pullback of the tautological line bundle $\gamma^1$ over $\C P^1$, which over the point 
$[z_1, z_2] \in \C P^1$ has fiber that is the linear subspace of $\C^2$ containing $(z_1, z_2)$.  Furthermore, the first Chern class of $\gamma^1$ generates
$H^2(\C P^1)$ by \cite[Theorem 14.4]{MS}.  But $\C P^1\cong S^2$, and we know there are maps $f_k:S^2\longrightarrow S^2$ of any integer degree $k$. 
By naturality of characteristic classes, the pullback 
bundle $L_k=f^*_k\gamma^1$ must then have Chern class $k c^1(\gamma)$. Therefore, given any $k\in H^2(\C P^1)\cong \Z=\langle c^1( \gamma^1) \rangle$,
the class $k$ is the Chern class of the line bundle $L_k$, which is a subbundle of $\epsilon^2$.  Let $P_k$ be the matrix representing the projection 
operator from $\epsilon^2$ to $L_k$. Over each point, the projection has one eigenvalue equal to $1$ and one equal to $0$, so  
$P_k$ has  characteristic polynomial $\lambda^2 - \lambda$. All projections are normal operators, and the two eigenspace bundles of  
$P_k$ are $E_1=L_k$ and $E_0=L_k^{\perp}$. From our discussion just above, $[\theta(D,P_k)]=k\oplus -k\in H^2(\C P^1)\oplus H^2(\C P^1)$. 

It now follows from these computations and from Proposition \ref{P: Os}  that there are a countably infinite number of unitary equivalence classes
of normal matrices on $\C P^1$ with characteristic polynomial $\lambda(\lambda-1)$, indexed by the isomorphism classes of complex line bundles 
on $\C P^1$ or, equivalently, their Chern classes.
\end{example}

\begin{example}\label{E: 7.3}
 In this example, we construct explicitly an example of a nontrivial  ``twisted'' obstruction to unitary equivalence, i.e. a nonzero $[\theta(A,B)]$ for
 which the common characteristic polynomial has nontrivial monodromy of its roots.   

First, consider the tautological line bundle $\gamma^1$ over $\C P^1$, whose Chern class $c^1(\gamma^1)$ generates $H^2(\C P^1)\cong \Z$. We can 
consider $\gamma^1$ to be a subbundle of the trivial $\C^2$ bundle over $\C P^1$; in fact, the classifying map for $\gamma^1$ is the identity map
 $\C P^1\to \C P^1$, which assigns to each point in $\C P^1$ the complex line in $\C^2$ that it represents. Using the standard Hermitian structure 
 on $\C^2$, let $\gamma^\perp$ denote the perpendicular bundle to $\gamma^1$, and let $\nu:\C P^1\to \C P^1$ be the associated map taking
 $y\in \C P^1$ to the complex line orthogonal to the complex line represented by $y$. Then $\gamma^\perp=\nu^*\gamma^1$. As 
 $\gamma^1\oplus \gamma^\perp=\epsilon^2$, the trivial complex plane bundle, we have
 $0=c^1(\gamma^1\oplus \gamma^\perp)=c^1(\gamma^1)+c^1(\gamma^\perp)$, so $c^1(\gamma^\perp)=-c^1(\gamma^1)$. By the naturality of
 Chern classes, we see that $\nu: \C P^1\to \C P^1$ must have degree $-1$. Furthermore, $\nu$ must be a homeomorphism because every linear
 subspace of $\C^2$ has a unique orthogonal subspace. 

Let $X$ be the quotient space of $I\times \C P^1$ by the identification $(1,y)\sim (0,\nu(y))$. Notice that $X$ has the structure of a $\C P^1$ bundle 
over $S^1$. Let $p:X\to S^1$ be the projection. From the long exact sequence of the fibration, we must have $\pi_1(X)\cong \pi_1(S^1)\cong \Z$.
We can similarly construct $X\times \C^2$ as the quotient space of $I\times \C P^1\times \C^2$ by the identification $(1,y,t)\sim (0,\nu(y),t)$. Thinking of
$E=X\times \C^2$ as the trivial $\C^2$ bundle over $X$, we can identify within $E$ a ``twisted double bundle'' that assigns two linear subspaces of
$\C^2$ to each point in $X$ but such that a trip around a generating loop of $\pi_1(X)$ keeping track of these lines results in interchanging the two
subspaces. In fact, to the image of  each point $(z,y)\in I\times \C P^1$, we assign the complex line represented by $y$  and  the orthogonal subspace
to the line represented by $y$. While this is clearly well defined on $I\times \C P^1$, it is also well defined on $X$ by our construction, as the
quotient identifies two points corresponding to orthogonal lines. 

Choose a base point $z_0\in S^1$. Over $p^{-1}(z_0)\cong \C P^1$, our ``double bundle'' reduces to copies of $\gamma^1$ and $\gamma^\perp$.
Let us assign to one of these bundles one square root of $z_0$ (identifying $S^1$ with the standard unit circle in $\C$) and to the other bundle the 
other square root of $z_0$. We can continuously extend these assignments, assigning the two square roots of $z$ to the two orthogonal bundles on
$p^{-1}(z)$ for each $z\in S^1$. Of course each time we loop around the full circle, the two square roots are interchanged, but, by construction, so 
are the bundles! Therefore, we achieve a well-defined continuous global assignment $\pm\sqrt{z}$  to  the bundles  over $p^{-1}(z)$. Now, at each 
point  $x\in X$, there is a unique matrix $B(x)\in M_2(\C)$ whose eigenspaces correspond to the complex lines in $\C^2$ given by restricting our 
double bundle to $x$ and whose eigenvalues are the values in $S^1$ given by our assignment\footnote{Suppose we choose vectors $v,w$ in our
designated eigenspaces with eigenvalues $\lambda_1\neq \lambda_2$. Then the standard basis vectors can be written in terms of $v$ and $w$ 
as $e_1=av+bw$ and $e_2=cv+dw$ for some $a,b,c,d\in \C$. But then we know exactly how $B(x)$ acts on $e_1$ and $e_2$, and this 
determines uniquely our matrix.}. Because our eigenvalues and eigenvectors vary continuously, so will $B(x)$, and this gives us a matrix 
$B\in M_2(C(X))$. The eigenspaces of $B$ are orthogonal at each point, so $B$ is normal, and it is clearly multiplicity free.

Consider the matrix
\[
A = p^*\begin{pmatrix} 0 & z \\ 1 & 0 \end{pmatrix}
\]
in $M_2(C(X))$; it follows from Example \ref{E: circle} and the fact that normality is preserved by pullbacks that $A$ is normal.
The characteristic polynomial of $A$ is $\mu=\lambda^2-z$, which is the same as the characteristic polynomial of $B$.  Because $A$ is a pullback matrix, 
the eigenspace bundles of the restriction of $A$ to $p^{-1}(z_0)$ are trivial. So, if we let $A_{z_0}$ and $B_{z_0}$ denote the restrictions of
$A$ and $B$ to $p^{-1}(z_0)$, then by Proposition \ref{P: chern}, we must have 
\[
[\theta(A_{z_0}, B_{z_0})]=c^1(\Hom(\epsilon^1, \gamma^1))\oplus c^1(\Hom(\epsilon^1, \gamma^\perp))
=c^1(\gamma^1)\oplus c^1(\gamma^\perp)\in H^2(\C P^1;\Z^2).
\]
This class is non-zero, so $A_{z_0}$ and $B_{z_0}$ are not unitarily equivalent over $p^{-1}(z_0)$. It follows that $A$ and $B$ cannot be
unitarily equivalent over $X$. 

This example demonstrates that the obstruction $[\theta(A,B)]$ can  be nontrivial when there is monodromy of eigenvalues. But this example
has the following additional amusing element: the group $H^2(X;\Z)$ is trivial, so any two normal matrices over $X$ with the same characteristic
polynomial with trivial monodromy \emph{are} unitarily equivalent by Theorem \ref{T: U equiv}. So here is a space where we have obstructions to 
unitary equivalence only when nontrivial monodromy of roots occurs. 

To verify the claim that $H^2(X;\Z)=0$, recall that $X$ is a $\C P^1$ bundle over $S^1$. In the Leray-Serre spectral sequence for the cohomology 
of $X$, the only $E_2$ term that  could contribute to $H^2(X)$ and that isn't evidently trivial is $E_2^{0,2}=H^0(S^1;\mc H^2(\C P^1))$. 
Here $\mc H^2(\C P^1)$ is the local coefficient system induced by the bundle structure.
As $H^2(\C P^1)\cong \Z$ and because we form $X$ by attaching $\{0\}\times \C P^1$ and $\{1\}\times \C P^1$ by a map of degree $-1$, this bundle is the 
bundle $\Z_\rho$, where $\rho:\pi_1(S^1)\cong \Z\longrightarrow \Aut(\Z)$ takes a generator of $\pi_1(S^1)$ to the nontrivial automorphism of $\Z$. But now give 
$S^1$ the standard CW structure with one $0$-cell $e^0$ and one $1$-cell $e^1$. Then, in the universal cover $\wt S^1\cong \R$, we have a natural 
CW structure with $0$- and $1$-cells $e^0_i, e^1_i$ for all $i\in \Z$. We can assume $\bd e^1_0=e^0_1-e^0_0$. If $\eta$ is a generator of 
$\pi_1(S^1)\cong \Z$, then $\pi_1(S^1)$ acts on the cellular chain complex $C_*(\wt S^1)$ by $\eta(e^j_i)=e^j_{i+1}$ for $j=0,1$.  The cohomology 
$H^*(S^1;\mc H^2(\C P^1))$ is then the cohomology of the cochain complex $C^*(S^1;\Z_\rho)=\Hom_{\mathbb{Z}[\Z]}(C_*(\td S^1), \Z_\rho)$, 
where we let  $\Z_\rho$ denote $\Z$ with the stated action as a $\pi_1(S^1)$ module. 

Let $f_a$ be the $0$-cochain such that $f_a(e^0_0)=a$. From the module structure, all elements of $C^0(S^1;\Z_\rho)$ have this form. 
We compute 

\begin{align*}
(df_a)(e^{1}_0)&=-f_a(\bd e^1_0)\\
&=-f_a(e^0_{1}-e^0_0)\\
&=-f_a(\rho e_0^0-e^0_0)\\
&=-(\rho f_a(e_0^0)-f_a(e_0^0))\\
&=-(\rho (a)-a)\\
&=-(-a-a)\\
&=2a;
\end{align*}
in the first line we follow the sign convention for coboundary operators determined by  \cite[Definition 10.1] {Do}.
Therefore, $df_a=0$ only if $f_a=0$. Thus there are no nontrivial cocycles in $C^0(S^1;\Z_\rho)$ and $H^0(S^1;\Z_\rho)=0$, as claimed.
\end{example}

\section{Further questions}

Our work here raises or leaves unanswered several questions for future research:

\begin{itemize}
\item In Section 7, we showed that if the common characteristic polynomial of multiplicity-free normal matrices of $A$ and $B$ globally factors into 
linear factors, then we can write our obstruction $[\theta(A,B)]$ in terms of the first Chern classes of the bundles $\Hom(E_i,F_i)$, where $E_i$ and $F_i$ are the respective eigenspace bundles of $A$ and $B$ with the same eigenvalue. This raises the question: more generally, when can we compute the obstruction $[\theta(A,B)]$ in terms of other known invariants? Similarly, are there effective computational algorithms for determining when $[\theta(A,B)]=0$, given $A$ and $B$?

\item We  also saw in Section 7, particularly in Examples \ref{E: sphere} and \ref{E: 7.3}, that not every element of $H^2(X; \Z^n_\rho)$ can be realized as an obstruction class $[\theta(A,B)]$. So, which cohomology classes can be realized as obstructions? By Proposition \ref{P: Os}, an answer to this question would determine the number of unitary equivalence classes with a given multiplicity-free characteristic polynomial. 

\item What can be said about normal matrices that are not multiplicity free? Such matrices are nongeneric, in the sense that any such matrix can be made multiplicity free by an arbitrarily small (in your favorite reasonable sense) perturbation.  As our example in the introduction suggests, non-multiplicity-free normal matrices turn out to be much more complicated than multiplicity-free ones, even if the underlying topological space is contractible.  Therefore the algebraic topological methods that we employ in this paper are unlikely to shed much light on non-multiplicity-free normal matrices, and thus other techniques, perhaps involving algebraic geometry, will be needed.

\item What is $H^2(B_n; \mathbb{Z})$?  A concrete description of this group might shed light on our obstruction.  Also, what additional information
can be discovered about the fiber bundles $p: E_n \longrightarrow B_n$?   For example, is there a structural group and an associated principal bundle?
\end{itemize}

\bibliographystyle{amsplain}

\vskip 24pt

\parindent 0pt
Greg Friedman \newline
Box 298900 \newline
Texas Christian University\newline
Fort Worth, TX 76129\newline
g.friedman{$@$}tcu.edu

\vskip 24pt
Efton Park \newline
Box 298900 \newline
Texas Christian University\newline
Fort Worth, TX 76129\newline
e.park{$@$}tcu.edu

\end{document}